\documentclass[11pt, oneside]{amsart}

\usepackage{tabularx}
\usepackage{amssymb} \usepackage{amsfonts} \usepackage{amsmath}
\usepackage{amsthm} \usepackage{epsfig, subfig}
\usepackage{ amscd, amsxtra, latexsym}
\usepackage[all]{xy}
\usepackage{caption}
\usepackage{enumerate}
\usepackage{color}

\newtheorem{lemma}{Lemma}[section]
\newtheorem{thm}[lemma]{Theorem}
\newtheorem{prop}[lemma]{Proposition}
\newtheorem{cor}[lemma]{Corollary}
\newtheorem{conj}[lemma]{Conjecture}

\theoremstyle{definition}
\newtheorem{defn}[lemma]{Definition}

\newtheorem{quest}[lemma]{Question}

\newtheorem{rem}[lemma]{Remark}

\theoremstyle{definition}

\definecolor{darkgreen}{cmyk}{1,0,1,.2}

\newcommand{\g} {\ensuremath {\gamma}}

\newcommand{\N}{\ensuremath {\mathbb{N}}}
\newcommand{\R} {\ensuremath {\mathbb{R}}}

\newcommand{\Z} {\ensuremath {\mathbb{Z}}}

\newcommand{\matP} {\ensuremath {\mathbb{P}}}

\newcommand{\calC} {\ensuremath {\mathcal{C}}}

\newcommand{\calP} {\ensuremath {\mathcal{P}}}

\newcommand{\calM} {\ensuremath {\mathcal{M}}}
\newcommand{\calH} {\ensuremath {\mathcal{H}}}

\newcommand{\calA} {\ensuremath {\mathcal{A}}}

\newcommand{\calI} {\ensuremath {\mathcal{I}}}

\newcommand{\calJ} {\ensuremath {\mathcal{J}}}

\newcommand{\calPS} {\ensuremath {\mathcal{PS}}}

\newcommand{\tilM}{\ensuremath{{\widetilde{M}}}}
\newcommand{\tilN}{\ensuremath{{\widetilde{N}}}}

\usepackage{hyperref}

\title{Contracting elements and random walks}
\author[A. Sisto]{Alessandro Sisto}
\address{ETH, Z\"{u}rich, Switzerland}
\email{sisto@math.ethz.ch}

\subjclass[2010]{Primary 20F65; Secondary 60G50}

\begin{document}

\begin{abstract}
We define a new notion of contracting element of a group and we show that contracting elements coincide with
hyperbolic elements in relatively hyperbolic groups, pseudo-Anosovs in mapping class groups, rank one isometries
in groups acting properly on proper $CAT(0)$ spaces, elements acting hyperbolically on the Bass-Serre tree in graph
manifold groups. We also define a related notion of weakly contracting element, and show that those coincide with
hyperbolic elements in groups acting acylindrically on hyperbolic spaces and with iwips in $Out(F_n)$, $n\geq 3$. We show that each weakly contracting element is contained in a hyperbolically embedded elementary subgroup, which allows us to answer a problem in \cite{DGO}.
We prove that any simple random walk in a non-elementary finitely generated subgroup containing
a (weakly) contracting element ends up in a non-(weakly-)contracting element with exponentially decaying probability.
\end{abstract}

\maketitle

\section{Introduction}
A Morse element in a group $G$ is an element $h$ such that $H=\langle h \rangle$ is undistorted in $G$ and any quasi-geodesic with endpoints on $H$ stays within bounded distance from $H$, and the bound depends on the quasi-isometry constants only. Examples of Morse elements include infinite order elements in hyperbolic groups \cite{Gro-hyp}, hyperbolic elements (of infinite order) in relatively hyperbolic groups \cite{DS-treegr}, pseudo-Anosovs in mapping class groups \cite{Be-asgeommcg}, etc. See \cite{DMS-div} for further details and examples. Indeed, in all mentioned cases a stronger property than being Morse holds. Namely, all elements are \emph{contracting} with respect to a suitable collection, called \emph{path system}, of quasi-geodesic paths in the respective groups. Roughly speaking, the main property that a contracting element $g$ satisfies is the existence of a map $\pi_g$ from the group onto $\langle g\rangle$ such that if two points $x,y$ have far away projections then all special paths from $x$ to $y$ pass close to $\pi_g(x)$ and $\pi_g(y)$. (This definition is less general than the one we will give, which is stated in terms of group actions on a metric space.) Related properties have been considered in \cite{BF,Be-asgeommcg,A-K,Si-proj}. These are the examples we can provide, see Section \ref{eg}.

\begin{thm}\label{examples}
 Let $G$ be a relatively hyperbolic group (resp. mapping class group, group acting properly by isometries on a proper $CAT(0)$ space, graph manifold group). Then there exists a path system for $G$ such that $g\in G$ is contracting if and only if $g$ is hyperbolic (resp. is pseudo-Anosov, acts as a rank one isometry, acts hyperbolically on the Bass-Serre tree).
\end{thm}

Notice that not all graph manifolds are $CAT(0)$ \cite{L,BK}.

We also define a more general notion, that of being \emph{weakly contracting}, which will be sufficient for our applications. This concept is defined using the notion of weak path system described in Definition \ref{def:wps}.

\begin{thm}[Proposition \ref{acyl}, Proposition \ref{outfn}]
\label{examples2}
 Let $G$ be a group acting acylindrically on a hyperbolic space $X$, and let $\calPS$ be the collection of all geodesics of $X$. Then $g\in G$ is weakly contracting for the weak path system $(X,\calPS)$ if and only if it acts hyperbolically on $X$.
\par
There exists a weak path system for $Out(F_n)$, $n\geq 3$, such that $g\in Out(F_n)$ is weakly contracting if and only if it is fully irreducible.
\end{thm}

The disadvantage of the weak version of contractivity is that it is not as straightforward to show that weakly contracting elements are Morse, even though this is true \cite{Si-hypembqconv}. More importantly, even though we will not go into this here, weakly contracting elements give worse lower bounds on the divergence function.
\par
We will treat all examples mentioned in Theorems \ref{examples} and \ref{examples2} from a common perspective. As it turns out, this approach is closely related to the approach relying on the notion of hyperbolically embedded subgroup developed in \cite{DGO}.

\begin{thm}[Theorem \ref{hypemb}]\label{hypembintro}
 Each weakly contracting element $g$ (for some action) is contained in a hyperbolically embedded elementary subgroup $E(g)$. Viceversa, every infinite order element contained in a virtually cyclic embedded subgroup is weakly contracting for an appropriate action.
\end{thm}

Osin further showed \cite{Os-acyl} that the notion of containing a (nondegenerate) hyperbolically embedded subgroup is equivalent to other notions that had been previously defined \cite{BeFu-wpd,Ha-isomhyp}.

The subgroup $E(g)$ is defined in a natural way in terms of projections, see Corollary \ref{e(g)}. Notice that in view of Theorem \ref{examples} we obtain new examples of hyperbolically embedded subgroups, i.e. $E(g)$ where $g$ acts as a rank one isometry on a proper $CAT(0)$ space, answering a question of Dahmani-Guirardel-Osin.
\par
Our main result is that (weakly) contracting elements are generic, from the random walks viewpoint. Let $G$ be a group and $S\subseteq G$ a finite subset. Let $W_n(S)$ be the set of words of length $n$ in the elements of $S$ and their inverses. A \emph{simple random walk} supported on $\langle S\rangle$ is a sequence of random variables $\{X_n\}$ taking values in $G$ and with laws $\mu_n$ so that for each $g\in G$ and $n\in \N$ we have
$$\mu_n(\{g\})=\frac{|\{w\in W_n(S):w \text{\ represents\ }g\}|}{|W_n(S)|}.$$

\begin{thm}\label{randomwalksintro}
  Let $G$ be a group equipped with a path system (resp. weak path system) and let $H<G$ be a non-elementary finitely generated subgroup containing a contracting (resp. weakly contracting) element. Then the probability that a simple random walk supported on $H$ gives rise to a non-contracting (resp. non-weakly-contracting) element decays exponentially in the length of the random walk.
\end{thm}

Roughly speaking, the theorem states that the probability that a long word written down choosing randomly generators of $H$ represents a non-contracting element is very small, and in fact exponentially decaying in the length of the word.
\par
For mapping class groups this was known already \cite{Ma-expdecay}, and related results can be found in \cite{Ri1, Ri2, Ma1, MS}. We emphasize that we can give a self-contained proof of the theorem above modulo a classical result of Kesten.

Finally, we point out two sample applications of the theorem to give examples of how it can be used to show the existence of contracting elements with additional properties.

\begin{cor}
Let $G_1,G_2$ be a non-elementary group supporting the path system (resp. weak path system) $(X,\calPS)$ and containing a contracting (resp. weakly contracting) element.
 \begin{itemize}
  \item For any isomorphism $\phi:G_1\to G_2$ there exists a (weakly) contracting $g\in G$ so that $\phi(g)$ is also (weakly) contracting.
  \item If $H$ is an amenable group and $\phi:G_1\to H$ is any homomorphism, for any $h\in \phi(G_1)$ there exists a (weakly) contracting $g\in \phi^{-1}(h)$.
 \end{itemize}
\end{cor}

\subsection*{Outline} Section \ref{defcontr} contains the definitions we will use throughout. Section \ref{eg} is devoted to the proof of Theorems \ref{examples} and \ref{examples2}. In Section \ref{hypemb:sec} we will show Theorem \ref{hypembintro}. Section \ref{constr:sec} contains our main technical tool (Lemma \ref{hg}), which gives us a way of constructing many contracting elements once we are given just one of them. Finally in Section \ref{mainproof} we will show our main result. In the Appendix we will sketch the proof of a ``Weak Tits Alternative'' not relying on Theorem \ref{hypembintro} and \cite{DGO}.

\subsection*{Acknowledgements} The author would like to thank Cornelia Dru\c{t}u and Denis Osin for helpful conversations and suggestions and Ric Wade for clarifications on $Out(F_n)$. The author was partially funded by the EPSRC grant "Geometric and analytic aspects of infinite groups".

\section{Definitions and first properties}
\label{defcontr}
Let $X$ be a metric space.
\begin{defn}
 A \emph{path system} $\calPS$ in $X$ is a collection of $(\mu,\mu)$-quasi-geodesics in $X$, for some $\mu$, such that
\begin{enumerate}
 \item any subpath of a path in $\calPS$ is in $\calPS$,
 \item all pairs of points in $X$ can be connected by a path in $\calPS$.
\end{enumerate}

Elements of $\calPS$ will be called \emph{$\calPS$-special paths}, or simply special paths if there is no ambiguity on $\calPS$.

\end{defn}

Fix a path system $\calPS$ on the metric space $X$.

\begin{defn}\label{maindef}
 A subset $A\subseteq X$ will be called $\calPS$-\emph{contracting} with constant $C$ if there exists $\pi_A=\pi:X\to A$ such that
\begin{enumerate}
 \item $d(\pi(x),x)\leq C$ for each $x\in A$,
 \item for each $x,y\in X$, if $d(\pi(x),\pi(y))\geq C$ then for any $\calPS$-special path $\delta$ from $x$ to $y$ we have $d(\delta,\pi(x)),d(\delta,\pi(y))\leq C$.
\end{enumerate}
The map $\pi$ will be called \emph{$\calPS$-projection} on $A$ with constant $C$.
\end{defn}

We point out a relation with properties that appeared in the literature in several contexts, see \cite{BF,A-K, Si-proj}.

\begin{lemma}\emph{\cite[Lemma 4.24]{Si-proj}}
\label{4.24}
 Let $X$ be a geodesic metric space and $A\subseteq X$. Denote by $\calPS$ the collection of all geodesics in $X$. Suppose that the map $\pi:X\to A$ satisfies the following properties for some $C$:
\begin{itemize}
 \item $d(x,\pi(x))\leq d(x,A)+C$ for each $x\in X$,
 \item $diam(\pi(B_d(x)))\leq C$ for each $x\in X$, where $d=d(x,A)$.
\end{itemize}
Then $\pi$ is a $\calPS$-projection with constant depending on $C$ only.
\end{lemma}

The following lemma will be used several times.

\begin{lemma}\label{loccnst}
 Let $\pi$ be a $\calPS$-projection with constant $C$ on $A\subseteq X$. Then
\begin{enumerate}
 \item whenever $\delta$ is a special path we have $diam(\pi(\delta))\leq diam(\delta\cap N_C(A))+2C$ and more specifically $diam(\pi(\delta))\leq C$ if $\delta\cap N_C(A)=\emptyset$.
\par
Also, there exists $k=k(\calPS)$ such that
 \item for each $x,y\in X$, $d(\pi(x),\pi(y))\leq kd(x,y)+k$,
 \item for each $x\in X$ we have $diam(\pi(B_r(x)))\leq C$ for $r=\dfrac{d(x,A)}{k}-k$,
 \item for each $x\in X$, $d(x,\pi(x))\leq kd(x,A)+k$.
\end{enumerate}
\end{lemma}

\begin{proof}
 Item $1)$ is clear from the second projection property.
\par
Item $2)$ follows from special paths being quasi-geodesics and the fact that either $d(\pi(x),\pi(y))\leq C$ or any special path from $x$ to $y$ passes $C$-close to $\pi(x),\pi(y)$.
\par
Item $3)$ holds in view of item $1)$ because, for $c$ large enough, for each $y\in B_r(x)$ there is a special path connecting $x$ to $y$ and not intersecting $N_C(A)$.
\par
In order to show item $4)$, consider a special path $\g$ from $x$ to some $y\in A$ with $d(x,y)\leq d(x,A)+1$. If $d(y,\pi(x))\leq C$, we are done. Otherwise, $\g$ contains a point $x'$ such that $d(x',\pi(x))\leq C$. In particular,
$$d(x,\pi(x))\leq d(x,x')+C\leq \mu d(x,y)+\mu+C\leq \mu d(x,A)+\mu+C+1,$$
and we are done for $c$ large enough.
\end{proof}

\begin{lemma}\label{dichot}
 If $k$ is as in the previous lemma, $A,B\subseteq X$ are $\calPS$-contracting subsets with constant $C$ and $x\in X$, then
$$\min\{d(\pi_A(x),\pi_A(B)),d(\pi_B(x),\pi_B(A))\}\leq C'.$$
where $C'=kC+k+C$.
\end{lemma}

\begin{proof}
 If $\max\{d(\pi_A(x),\pi_A(B)),d(\pi_B(x),\pi_B(A))\}\leq kC+k+C$ there is nothing to prove, so, up to swapping $A$ and $B$, suppose that $d(\pi_A(x),\pi_A(B))> kC+k+C$. Let $\delta$ be a special path from $x$ to some $x'\in B$, and let $p$ be the first point in $\delta\cap N_C(A)$. If we show that there is no point $q\in\delta\cap N_C(B)$ before $p$, we are done by Lemma~\ref{loccnst}$-(4)-(1)$. In fact, if there was such a point $q$, again by Lemma~\ref{loccnst}$-(1)-(4)$ we would have
$$d(\pi_A(x),\pi_A(B))\leq d(\pi_A(x),\pi_A(q))+d(\pi_A(q),\pi_A(B))\leq C+kC+k.$$
\end{proof}

We are ready for our main definitions.

\begin{defn}
 Let $G$ be a group. A \emph{path system} $(X,\calPS)$ on the group $G$ is a \emph{proper} action of $G$ on the metric space $X$ preserving the path system $\calPS$ on $X$.
\par
An infinite order element $g$ of $G$ will be called $\calPS$-contracting if for some (and hence any) $x_0\in X$ and any integer $n$
\begin{enumerate}
 \item the orbit of $x_0$ is a quasi-geodesic,
 \item there exists a $g$-invariant contracting set $A\ni x_0$, called \emph{axis}, so that $g$ acts coboundedly on $A$.
\end{enumerate}
\end{defn}

The reason why we do not just take $A$ to be the orbit is that we will be interested in contracting elements with translation length much larger than the contraction constant of their axes, and we will achieve this by taking powers while keeping the same axis.

Slightly extending a notion from \cite{BeFu-wpd} we say that, given a group $G$ acting on the metric space $X$, the element $g\in G$ is \emph{WPD} if the orbits of $\langle g \rangle $ are Morse quasi-geodesics and for every $R\geq 0$ and $x\in X$ there exists $N=N(R,x)$ such that
$$|\{h\in G: d(x,h(x))\leq R, d(hg^N x, g^N x)\leq R\}|<+\infty.$$
\par
We will need the following consequence of the WPD property.

\begin{lemma}\emph{(cf. \cite[Proposition 6$-(2)$]{BeFu-wpd})}
\label{swpd}
 Suppose that $G$ acts on $X$ and $g\in G$ is WPD. Then for each $x\in X$, $R\geq 0$ there exists $L$ so that
$$\{h\in G:d(x,hx)\leq R, diam(N_R(h\langle g \rangle x)\cap \langle g \rangle x)\geq L\}$$
is a finite set.
\end{lemma}

The meaning of the lemma is illustrated in Figure \ref{swpd:pic}. Roughly speaking, the definition of WPD gives us that there are finitely many elements $h$ so that for some large $N$ we have the situation depicted for $N'=N$. The lemma gives us the same condition for all large enough $N$ and without the requirement $N'=N$.
\begin{figure}[ht]
\hspace{-.5cm}
\begin{minipage}[bl]{.45\textwidth}
\includegraphics[width=\textwidth]{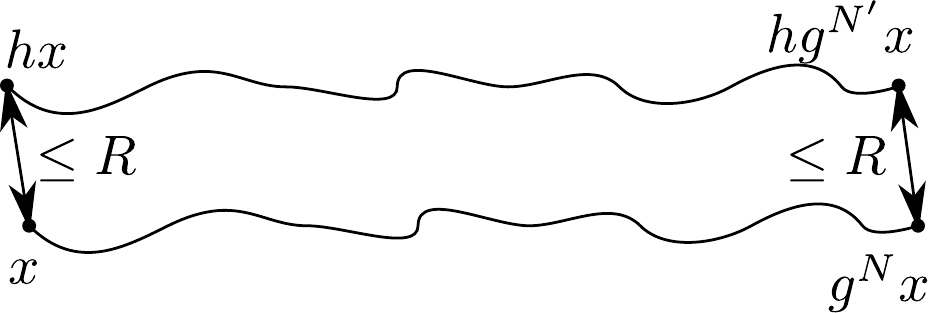}
\caption{}\label{swpd:pic}
\end{minipage}
\begin{minipage}[br]{.55\textwidth}
\includegraphics[width=\textwidth]{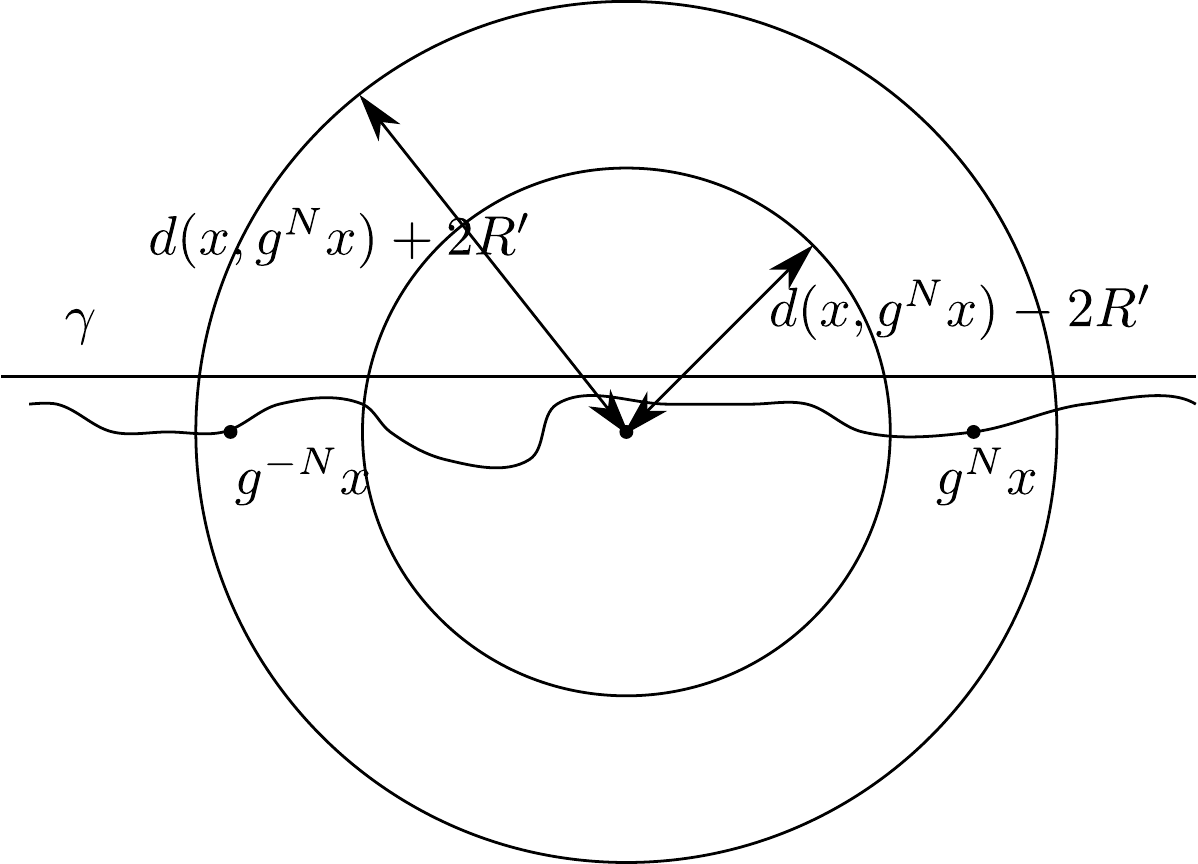}
\caption{}\label{swpd2}
\end{minipage}
\end{figure}

\begin{proof}
Fix $x,R,g$. First of all notice that, as $\langle g \rangle x$ is Morse, there exists $R'\geq R$ depending on $R$ (and the data we fixed) so that if $d(x,hx)\leq R, d(g^{N}x,h\langle g \rangle x)\leq R$ then for each $0\leq n\leq N$ we have $d(g^nx,h\langle g \rangle x)\leq R'$. We claim that it suffices to show that if
$$d(x,hx)\leq R, d(g^{N}x,hg^{N'}x)\leq R'\ \ (*)$$
for some $N,N'$ then either $d(g^{N}x,hg^{N}x)\leq C(R)$ or $d(g^{N}x,hg^{-N}x)\leq C(R)$, for some constant $C(R)\geq R$. In fact, if $h,h'$ are as in the second case then $hh'$ is easily seen to satisfy $d(hh'x,x)\leq 2R$, $d(g^{N}x,hh'g^{N}x)\leq 2C(R)$. So, the cardinality of the set as in the statement is at most the cardinality of
$$\{h\in G: d(x,hx)\leq 2R, d(hg^N x, g^N x)\leq 2C(R)\},$$
where for $L$ large enough we can choose $N=N(C(R))$. By definition of $N(C(R))$ this set is finite, as required.
\par
Now, suppose that $(*)$ holds. Then we have
$$|d(g^{N'}x,x)-d(g^Nx,x)|\leq 2R'.$$
So, we just need to show that if $N'$ satisfies this condition then $|N-N'|\leq K(R)$. In order to show this notice that, as the orbits of $g$ are Morse quasi-geodesics, there exists a geodesic $\gamma$ so that $d(g^N, \gamma),d(g^{N'},\gamma),d(x,\gamma)\leq D$ (where $D$ does not depend on $N, N'$), see Figure \ref{swpd2}. In particular, if $N'$ satisfies the condition then it is contained in a ball of radius, say, $2R'+10D$ around $g^{N}x$ or in a ball of the same radius around $g^{-N}x$. The existence of $K(R)$ then follows from the orbit of $g$ being a quasi-geodesic.
\end{proof}

\begin{defn}\label{def:wps}
 Let $G$ be a group. A \emph{weak path system} $(X,\calPS)$ on the group $G$ is any action of $G$ on the metric space $X$ which preserves the path system $\calPS$ on $X$.
\par
An infinite order element $g$ of $G$ will be called weakly $\calPS$-contracting if it is WPD and some (hence any) orbit of $g$ is contained in a contracting $g$-invariant set $A$ on which $g$ acts coboundedly.
%
%
\end{defn}

Notice that we weakened the requirements for the action, by dropping the properness assumption, and ``compensated'' for this by requiring that the action is only proper ``in the direction'' of $g$.


\begin{lemma}\label{morse}
\hskip 0cm
\begin{enumerate}
 \item Special paths which are $\calPS$-contracting with constant $C$ are Morse with constants depending on $C$ only.
 \item Contracting elements are Morse.
\end{enumerate}
\end{lemma}

\begin{proof}
The first part can be proven in the same (actually, slightly simpler) way as the second part, so we will spell out the proof for the second part only.
\par
 Let $G, X, \calPS, x_0$ be as in the definition above, and let $g$ be a contracting element. The fact that $\langle g \rangle x_0$ is a quasi-geodesic easily implies that $\langle g\rangle$ is a quasi-geodesic in $G$.
\par
In order to show that $g$ is Morse we now need to show the following. Let $\alpha$ be a $(\mu,c)$-quasi-geodesic in $G$ connecting $1$ to $g^n$ and let $\g$ be a special path from $x_0$ to $g^nx_0$. Then for each $h\in \alpha$ we have $d(\gamma x_0,hx_0)\leq K$, where $K$ depends on $g,\mu,c$ and the action of $G$. A constant depending on the said data will be referred to as universal.
\par
Let $\pi$ be a projection on $\g$ with constant $C$. Increase $C$ suitably and define $\pi':G\to\langle g\rangle$ in such a way that $d(\pi'(x)x_0,\pi(x))$ is bounded by $C$.
Let $\rho$ be a universal constant such that
$$d(\pi'(h),\pi'(h'))\leq \rho d(\pi(hx_0),\pi(h'x_0))+\rho$$
for each $h,h'\in G$. Let $r\in\N$ be the least integer so that $d(\alpha(j),\alpha(j+r))> \mu^2(\rho C+1)+c$ for each integer $j$ such that $j,j+r$ are in the domain of $\alpha$. Whenever $h=\alpha(j),h'=\alpha(j+r)$ we will say that $h,h'$ are consecutive. Notice that there is a universal bound $L$ on $d(hx_0,h'x_0)$ whenever $h,h'\in \alpha$ are consecutive.
Suppose that $h_0,\dots, h_n\in \alpha$ is a maximal chain of consecutive points such that $d(hx_0,\g)\geq kL+k^2$, for $k$ as in Lemma \ref{loccnst}. Notice that we can bound the distance of $h_0,h_n$ from $\g$ again by a universal constant. We wish to show that $n$ cannot be arbitrarily large. In fact, by Lemma \ref{loccnst}$-(3)$ we have $d(\pi(h_i(x_0)),\pi(h_{i+1}(x_0)))\leq C$, and hence $d(\pi(h_0(x_0)),\pi(h_{n}(x_0)))\leq nC$ so that $d(\pi'(h_0),\pi'(h_n))\leq n\rho C+\rho$. For $n$ large enough and in view of $d(h_i,h_{i+1})> \mu^2(\rho C+1)+c$ we get
$$d(h_0,h_n)\geq nr/\mu-c\geq\sum (d(h_i,h_{i+1})-c)/\mu^2-c>$$
$$n\rho C+\rho+d(h_0,\pi'(h_0))+d(h_n,\pi'(h_n))\geq d(h_0,h_n),$$
which is a contradiction.
We can bound $d(p,\g)$ in terms of $n,L,r$, so we are done.
\end{proof}


\section{Examples}\label{eg}

\subsection{Relatively hyperbolic group}


In this subsection we study contracting elements in relatively hyperbolic groups.

We will use the following well-known results, a proof of which is given for the convenience of the reader. As usual, we denote (almost) closest-point projections on $P$ by $\pi_{P}$.
\begin{lemma}
 \label{qconvcontr}
Suppose that $X$ is hyperbolic and $P\subseteq X$ is quasi-convex. Then there exists $C$, depending on the quasi-convexity constants only, so that for each $x,y\in X$ if $d(\pi_P(x), \pi_P(y))\geq C$ then $d([x,y],\pi_P(x)), d([x,y],\pi_P(y))\leq C$ for each geodesic $[x,y]$ from $x$ to $y$.
\end{lemma}

\begin{proof}
Set $\pi=\pi_P$ and let $\delta$ be a hyperbolicity constant for $X$.
As quadrangles in $X$ are $2\delta$-thin, we have that $[\pi(x),\pi(y)]$ is contained in the $2\delta$-neighbourhood of $[x,y]\cup[x,\pi(x)]\cup [y,\pi(y)]$. Due to our choice of $\pi$, points $[\pi(x),\pi(y)]$ farther than $K$ from the endpoints are not $2\delta$-close to $[x,\pi(x)]\cup [y,\pi(y)]$, where $K$ depends on the quasi-convexity constants of $P$. Hence, for $d(\pi(x),\pi(y))>2K+1$, there is a point on $[x,y]$ which is $2\delta$-close to the point on $[\pi(x),\pi(y)]$ at distance $K+1$ from $\pi(x)$, and so we have $d(\alpha,\pi(x))\leq K+2\delta+1$.
\end{proof}

We will say that an infinite order element of a relatively hyperbolic group is \emph{hyperbolic} if it not conjugate into any $H_i$.

\begin{prop}\label{rhcontr}
Let $G$ be a relatively hyperbolic group and let $\calPS$ be the collection of the geodesics in $Bow(G)$. Then an element of $G$ is contracting for the path system $(Bow(G),\calPS)$ if and only if it is hyperbolic.
\end{prop}

\begin{proof}
If an element is finite order or conjugate into a peripheral subgroup then clearly its orbit is not a bi-infinite quasi-geodesic.

On the other hand, hyperbolic elements are known to act hyperbolically on the Bowditch space, so that each orbit is a bi-infinite quasi-geodesic, see e.g. \cite[Definition 1-(3)]{Bow-99-rel-hyp}. Here is a proof of this fact. Let $g$ be a hyperbolic element. As $g$ has infinite order we only have to show that it cannot act parabolically. Indeed, we claim that if an element of $G$ acts parabolically on $Bow(G)$ then it fixes the parabolic point corresponding to some combinatorial horoball, and this easily implies that it must be conjugate into a peripheral subgroup. If $g$ acted parabolically fixing a point that is not the limit point of a combinatorial horoball, then there would be a geodesic ray $\g$ with $\g\cap G$ unbounded (we regard $G$ as a subset of $Bow(G)$) so that for each $n$ there is a point $h\in \g\cap G$ so that $d(h^{-1}g^ih,1)=d(g^ih,h)$ is bounded  by some constant $C$ which depends on the hyperbolicity constant of $Bow(G)$. But this is not possible because there are finitely many elements of $G$ in the ball of radius $C$ in $Bow(G)$ around $1$.

Finally, every geodesic in a given hyperbolic space (and hence every quasi-geodesic as quasi-geodesics are within bounded Hausdorff distance from geodesics) is contracting, see Lemma \ref{qconvcontr}.
\end{proof}

\subsection{Mapping class group}

In this subsection we will assume familiarity with the notion of curve complex, marking complex and hierarchy paths. We will use results from \cite{MM1,MM2}, see also \cite{Bow-cchyp}.

\begin{prop}
 Let $\calM(S)$ be the marking complex of the closed orientable punctured surface $S$ of genus $g$ with $p$ punctures satisfying $3g+p\geq 5$ and let $\calH(S)$ be the collection of all subpaths of hierarchy paths. An element of $MCG(S)$ is contracting for the path system $(\calM(S),\calH(S))$ if and only if it is pseudo-Anosov.
\end{prop}

Notice that the proposition and Lemma \ref{morse} give yet another proof that pseudo-Anosov elements are Morse \cite{Be-asgeommcg,DMS-div}.

\begin{rem}
 We remark that in view of the proof of \cite[Proposition 8.1]{BF} (see the sixth to last line) and Lemma \ref{4.24} the collection of all geodesics in Teichm\"uller space endowed with the Weil-Petersson metric has the property that all axes of pseudo-Anosov elements are contracting.
\end{rem}

It is well known that hierarchy paths are quasi-geodesics with uniform constant, and that any non-pseudo-Anosov element is not Morse as, up to passing to a power, it has infinite index in its centralizer.
\par
Recall from \cite{Be-asgeommcg} that a pair of points $(\mu_1,\mu_2)\in \calM(S)$ is $D$-transverse if for each $Z\subsetneq S$ we have $d_{C(Z)}(\mu_1,\mu_2)\leq D$.
The proposition follows directly from the lemma below and the fact that pairs of points lying on the orbit of a pseudo-Anosov element are $D$-transverse.

\begin{lemma}
 For each $D$-transverse pair $\mu_1,\mu_2$ each hierarchy path $[\mu_1,\mu_2]$ is $\calH(S)$-contracting with constant $C=C(D)$.
\end{lemma}

\begin{proof}
Let $B$ be such that if $x,y$ lie on the (subpath of a) hierarchy path $[p,q]$ with main geodesic $H$ then
\begin{enumerate}
 \item for each subsurface $Y\subseteq S$ we have $d_{\calC(Y)}(x,y)\leq d_Y(p,q)+B$,
 \item $d_{\calC(S)}(x,H)\leq B$,
 \item for any $\gamma\in H$ there exists $z$ such that $d_{\calC(S)}(x,H)\leq B$.
\end{enumerate}

 Define $\pi:\calM(S)\to [\mu_1,\mu_2]$ as follows. Consider $\mu\in\calM(S)$ and let $\gamma_\mu=\pi_S(\mu)$. Pick $\alpha_\mu\in H$ (the main geodesic for the given hierarchy path) such that $d_{C(S)}(\gamma_\mu,\alpha_\mu)=d_{C(S)}(\gamma_\mu,H)$. Finally, choose $\nu\in[\mu_1,\mu_2]$ in such a way that $\pi_{S}(\nu)$ is $B$-close to $\alpha_\mu$ and set $\pi(\mu)=\nu$.
\par
Pick any $\nu_1,\nu_2\in \calM(S)$. Suppose $d(\pi(\nu_1),\pi(\nu_2))\geq C(D)$, where $C(D)$ will be determined later. Notice that $d_{C(S)}(\alpha_{\nu_1},\alpha_{\nu_2})$ is bounded from below by a linear function of $C(D)$, by the distance formula and $D$-transversality, so that for $C(D)$ large enough we can assume $d_{C(S)}(\alpha_{\nu_1},\alpha_{\nu_2})\geq 100\delta$, where $\delta$ is the hyperbolicity constant of $C(S)$. Consider a hierarchy path $[\nu_1,\nu_2]$ (more precisely containing $\nu_1,\nu_2$, but this does not affect what follows) and let $H'$ be the main geodesic. We have that $H'$ contains points $\beta_1,\beta_2$ which are $B$-close to the projections of $\nu'_1,\nu'_2\in [\nu_1,\nu_2]$ on $C(S)$ and such that $d_{C(S)}(\alpha_{\nu_i}, \beta_i)\leq 10\delta$.

\begin{figure}[ht]
\centering
 \includegraphics[scale=0.6]{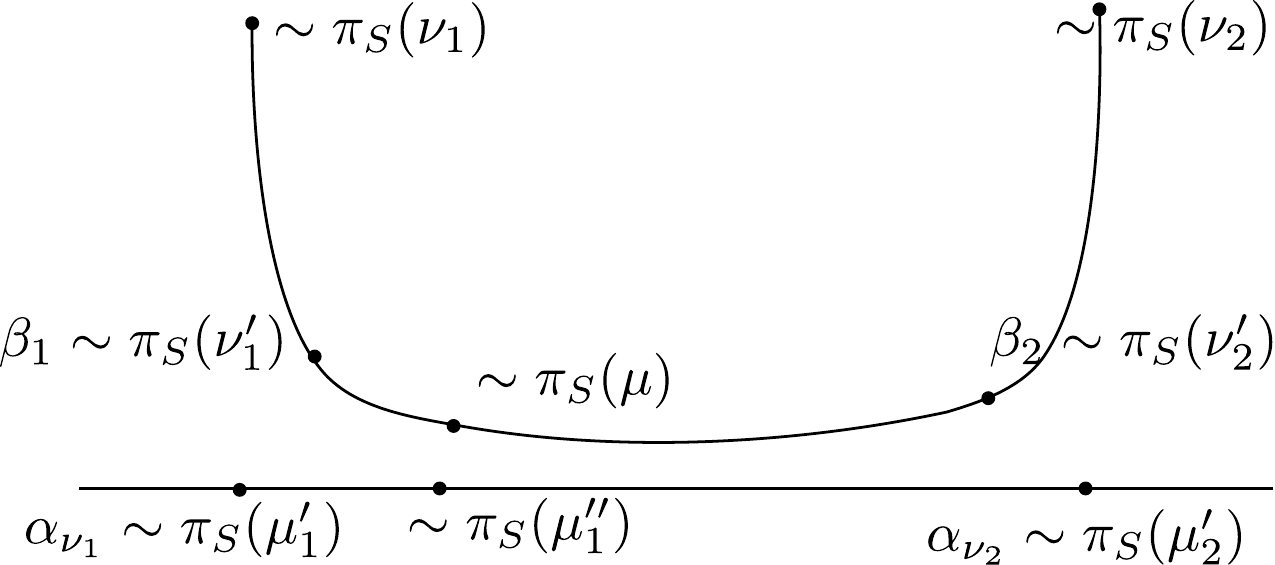}
\caption{}
\end{figure}
Our aim is to show that any point $\mu$ on $[\nu'_1,\nu'_2]$ such that $20\delta+2B+10\leq d_{\calC(S)}(\nu,\beta_1)\leq 20\delta+3B+20$ has the property that $d(\mu,\pi(\nu_1))\leq C(D)$ (an analogous property holds for $\nu_2$).
\par
In order to show this, we will now analyse the $K$-large domains for the pair $\nu'_1,\nu'_2$, where $K>3(D+B)$ (a $L$-large domain for a pair of markings $\rho_1,\rho_2$ is a subsurface $Y\subseteq S$ so that $d_{\calC(Y)}(\rho_1,\rho_2)\geq L$). First, we claim that there are no common $(K/3)$-large domains for the pairs $\mu'_1,\nu'_1$ and $\mu'_2,\nu'_2$, where $\mu'_i\in[\mu_1,\mu_2]$ projects $B$-close to $\alpha_{\nu_i}$.
\par
In fact, any $(K/3)$-large domain $X_i$ for the pair $\mu'_i,\nu'_i$ is contained in some $S\backslash\gamma_i$, where $\gamma_i$ is a simple closed curve appearing in a hierarchy path connecting $\pi_S(\mu'_i)$ and $\nu'_i$. This implies that $X_1\neq X_2$ as $d_{C(S)}(\gamma_1,\gamma_2)\geq 3$ for $C(D)$ large enough.
\par
We can now use the fact that there are no $D$-large domains for the pair $\mu'_1,\mu'_2$. Suppose $d_{C(Y)}(\nu'_1,\nu'_2)\geq K$, for some subsurface $Y\subsetneq S$. Then (at least) one of the following must hold: $d_{C(Y)}(\nu'_1,\mu'_1)\geq K/3$, $d_{C(Y)}(\mu'_1,\mu'_2)\geq K/3$ or $d_{C(Y)}(\nu'_1,\mu'_1)\geq K/3$. However, the second inequality does not hold by hypothesis. Hence, any $K$-large domain for $\nu'_1,\nu'_2$, and so any $(K+B)$-large domain for any pair of points on $[\nu'_1, \nu'_2]$, is a $(K/3)$-large domain for $\mu'_1,\nu'_1$ or $\mu'_2,\nu'_2$ (but not both).
\par
With a similar argument we get that for each $\nu\in[\nu'_1,\nu'_2]$ and each $\mu\in[\mu'_1,\mu'_2]$ all $(4K/3+D+2B)$-large domains $Y\subsetneq S$ for $\mu,\nu$ are $(K/3)$-large domains for $\mu'_1,\nu'_1$ or $\mu'_2,\nu'_2$. Choosing $\nu$ as above we have that $\pi_S(\nu)$ is far enough from $\beta_1$ and $\beta_2$ to guarantee that no $(4K/3+D+2B)$-large domain $Y\subsetneq S$ for $\nu,\mu$ is a $K$-large domain for $\mu'_i,\nu'_i$, where $\mu\in[\mu'_1,\mu'_2]$ is such that $20\delta+2B+10\leq d_{C(S)}(\mu,\mu'_i)\leq 20\delta+3B+20$. By the distance formula $d(\nu,\mu'_1)$ can be bounded in terms of $\delta$ and $B$, so we are done.
\end{proof}

\subsection{Graph manifolds}
A graph manifold is a compact connected $3$-manifold (possibly with boundary) which admits a decomposition into Seifert fibred surfaces, when cut along a collection of embedded tori and/or Klein bottles. In particular a graph manifold is a $3$-manifold whose geometric decomposition admits no hyperbolic part.
\par
Let $M$ be a graph manifold. It is known \cite{KL} that its universal cover $\tilM$ is bi-Lipschitz equivalent to the universal cover $\tilN$ of a \emph{flip} graph manifold $N$, that is to say a graph manifold with some special properties (among which a metric of nonpositive curvature). We will not need the exact definition of such manifolds, but we need to know that we can choose a bi-Lipschitz equivalence $\phi:\tilM\to\tilN$ that preserves the geometric components. In \cite{3mancones}, a family of paths $\calPS(N)$ in $\tilN$ has been defined, and those paths satisfy the following:
\begin{lemma}
\label{nice}
\hskip 0cm
\begin{enumerate}
 \item All paths in $\calPS(N)$ are bi-Lipschitz, with controlled constant;
 \item any subpath of a path in $\calPS(N)$ is again in $\calPS(N)$;
 \item if for $i=1,2$ $\alpha_i$ is a special path connecting some point in $X_{w_i}$ to some point in $X_{w'_i}$ and the vertex $v$
\begin{itemize}
 \item lies on the geodesic connecting $w_i,w'_i$, and
 \item $d(v,w_i),d(v,w'_i)\geq 2$,
\end{itemize}
  then $\alpha_1\cap X_v=\alpha_2\cap X_v$.
\end{enumerate}
\end{lemma}

Let $\calPS(M)=\{g\phi^{-1}(\g):g\in\pi_1(M),\g\in\calPS(N)\}$.
\begin{prop}\label{graphmancontr}
 $(\tilM,\calPS(M))$ is a path system for $\pi_1(M)$. An element of $\pi_1(M)$ is contracting if and only if it acts hyperbolically on the Bass-Serre tree of $M$.
\end{prop}

\begin{proof}
 The first part follows directly from Lemma \ref{nice}$-(1)-(2)$, so we can focus on the second part.
\par
Let us prove the ``if'' part. Let $g\in\pi_1(M)$ be an element acting hyperbolically on the Bass-Serre tree. Let $x_0\in X_{v}$ for some vertex $v$ of the Bass-Serre tree such that $\{g^iv\}$ is contained in a bi-infinite geodesic $\g$. Define $\pi:\tilM\to \{g^ix_0\}$ in the following way. For each $x\in\tilM$ define a vertex  $v(x)$ in the Bass-Serre tree such that $x\in X_{v(x)}$ and let $\pi'$ be the projection in the Bass-Serre tree on $\g$. Choose $i(x)$ so that $d(g^{i(x)}v,\pi'(v(x)))\leq d(v,gv)/2$. Finally, define $\pi(x)$ to be $g^{i(x)}x_0$.

Now, suppose that we have $x,y\in\tilM$ such that $d(\pi(x),\pi(y))$ is large enough. Then we can conclude that, say, $d(\pi'(v(x)),\pi'(v(y)))\geq 100$. We will now use Lemma \ref{nice} and the fact that $\phi$ preserves the geometric components to find a bound on the distance between any special path $\delta$ connecting $x,y$ and $\pi(x),\pi(y)$. By hypothesis, $\phi$ induces a simplicial isomorphism from the Bass-Serre tree of $M$ to that of $N$, which we will still denote by $\phi$. Suppose that $\delta=h\phi^{-1}(\alpha)$. Then it is clear from Lemma \ref{nice} that $\delta$ shares a subpath with any special path $h\phi^{-1}(\beta)$ which connects some point in $X_{g^{i(x)}v}$ to some point in $X_{g^{i(y)}v}$. More precisely, those paths coincide in $X_w$ whenever $w$ is ``well within'' $[\pi'(v(x)),\pi'(v(y))]$. In particular, we can choose $\beta$ so that the endpoints of $h^{-1}\phi(\beta)$ are $g^{i(x)}x_0, g^{i(y)}x_0$, so that any point in $h^{-1}\phi(\beta)$ is close to $\g$ as $\g$ is Morse \cite{KL}. It is now easy to see that points if $x',y'\in\delta\cap h^{-1}\phi(\beta)$ are so that $d(v(x'),g^{i(x)}v), d(v(y'),g^{i(y)}v)\leq 10$, then $x',y'$ are within uniformly bounded distance from $\pi(x),\pi(y)$.
\par
The ``only if'' part is easy. In fact, if $g\in\pi_1(M)$ is conjugate to some element in a vertex group, then it is clearly not even Morse in view of the fact that such groups are virtually products of a free group and $\Z$.
\end{proof}

\subsection{Groups acting acylindrically on hyperbolic spaces and \texorpdfstring{$Out(F_n)$}{}}
Recall that a group $G$ is said to act \emph{acylindrically} on the metric space $X$ if for all $d$ there exist $R_d,N_d$ so that whenever $x,y\in X$ satisfy $d(x,y)\geq R_d$ we have
$$|\{g\in G:d(x,gx)\leq d, d(y,gy)\leq d\}|\leq N_d.$$
Graph manifold groups act acylindrically on their Bass-Serre tree. However, we are not able to prove the equivalent of Proposition \ref{graphmancontr} for groups acting acylindrically on trees or hyperbolic spaces, and indeed it might not be true. Nonetheless, we have the following.
\begin{prop}\label{acyl}
 Let $G$ be a group acting acylindrically (by isometries) on the hyperbolic space $X$ and let $\calPS$ be the collection of all geodesics of $X$. Then $g\in G$ is weakly contracting for the weak path system $(X,\calPS)$ if and only if it acts hyperbolically on $X$.
\end{prop}

\begin{proof}
Given the following easy fact (Lemma \ref{qconvcontr}), the proof just requires unwinding the definitions: if $X$ is hyperbolic then all geodesics in $X$ are contracting.
\end{proof}

Bestvina and Feighn proved that the complex of free factors for $Out(F_n)$ is hyperbolic and that an element acts hyperbolically if and only if it is fully irreducible, in which case it also satisfies the WPD property \cite[Theorem 8.3]{BF2}. In particular we have the following.

\begin{prop}\label{outfn}
 Let $X$ be the complex of free factors of $Out(F_n)$ for some $n\geq 3$. Then $g\in Out(F_n)$ is weakly contracting for the weak path system $(X,\calPS)$ if and only it is fully irreducible.
\end{prop}

\begin{rem}
 For our applications we could also use the complexes constructed in \cite{BF1}.
\end{rem}

The present state of knowledge about the geometry of $Out(F_n)$ is not advanced enough to attempt imitating what we have done, say, for mapping class groups. It is not even known whether there is an algebraic characterization of Morse elements along the lines of the other cases, and hence the following is perhaps the most basic question arising at this point.
\begin{quest}
 Are all Morse elements of $Out(F_n)$ for $n\geq 3$ fully irreducible?
\end{quest}
 The converse is true by \cite{A-K}.
 A standard way to show that an (infinite order) element is not Morse is to show that (the cyclic group generated by) it has infinite index in its commensurator.
 There might be infinite order non-fully irreducible elements of $Out(F_n)$ which have finite index in their commensurators. So, if the answer to the previous question is negative, one might ask the following.
\begin{quest}
 Is it true that all elements of $Out(F_n)$ that have finite index in their commensurators are Morse?
\end{quest}

\subsection{Groups acting properly on proper \texorpdfstring{$CAT(0)$}{} spaces}

Recall that an isometry of a $CAT(0)$ space is \emph{of rank one} if it is hyperbolic and some (equivalently, every) axis of $g$ does not bound a half-flat.

\begin{rem}
In the case when $G$ is a right-angled Artin group and $X$ its standard $CAT(0)$ cube complex, the elements of $G$ that act as rank $1$ isometries coincide with those not conjugated into a join subgroup, see \cite{BC-raagcones}.
\end{rem}

\begin{prop}
 Let $G$ be a group acting properly by isometries on the proper $CAT(0)$ space $X$ and let $\calPS$ be the collection of all geodesics in $X$. Then $g\in G$ is contracting for the path system $(X,\calPS)$ if and only if it acts as a rank one isometry.
\end{prop}

\begin{proof}
 It is clear that $(X,\calPS)$ is a path system on $G$, and the ``only if'' part follows from the fact that if $g$ does not act as a rank one isometry then it is not Morse. Suppose that $g$ acts as a rank one isometry.
 By \cite[Theorem 5.4]{BF} the closest point projection on an axis $l$ of $g$ has the property that there exists $B$ so that each ball disjoint from $l$ projects onto a set of diameter at most $B$. Property $2)$ in Definition \ref{maindef} then follows from Lemma \ref{4.24}.
\end{proof}

\begin{rem}
 It is natural to ask when a $CAT(0)$ space admits a rank one isometry. This question is addressed by the Rank Rigidity Conjecture, formulated by Ballmann and Buyalo:
\begin{conj}\cite{BB-rrcconj}
 Let $X$ be a locally compact geodesically complete $CAT(0)$ space and $\Gamma$ be an infinite discrete group acting properly and cocompactly on $X$. If $X$ is irreducible, then $X$ is a higher rank symmetric space or a Euclidean building of dimension at least 2, or $\Gamma$ contains a rank one isometry.
\end{conj}
The conjecture is known to hold for Hadamard manifolds, see e.g. \cite{Ba-lect}, and $CAT(0)$ cube complexes \cite{CS-rrc}.
\end{rem}

\section{Contracting vs hyperbolically embedded}
\label{hypemb:sec}
Throughout this section we fix a weak path system $(X,\calPS)$ on the group $G$.

Suppose that $A$ is an axis of a weakly contracting element of $G$. Given a constant $b$, for $x,y,z\in A$ we say that $y$ is between $x$ and $z$ (with constant $b$) if any special path from $x$ to $z$ intersects $B_b(y)$. Similarly, for $h\in G$ and $hx,hy,hz\in hA$ we say that $hy$ is between $hx$ and $hz$ if $y$ is between $x$ and $z$.


\begin{lemma}[Projections are coarsely monotone]
Let $A$ be an axis of a weakly contracting element. Then there exist $b,K$ so that the following hold for each $h\in G$. Suppose, for $i=0,1,2$, that $y_i=\pi_A(x_i)$, for some $x_i\in hA$. Then if $y_1$ is between $y_0$ and $y_2$ with constant $b$ and $d(y_0,y_1),d(y_1,y_2)\geq K$, then $x_1$ is between $x_0$ and $x_2$, again with constant $b$.
\end{lemma}

\begin{proof}
Recall that special paths with endpoints on $A$ stay uniformly close to $A$, see Lemma \ref{morse}$-(1)$.
In particular, as $A$ is quasi-isometric to $\R$, there exists $b$ and $K_0$ so that if $x,y,z$ are points at reciprocal distance at least $K_0$ then there exists one of them which is between the other two with constant $b$.

We can assume $K\geq K_0$ and apply this for $\{x,y,z\}=\{x_0,x_1,x_2\}$.

Suppose by contradiction that $x_2$ is between $x_0$ and $x_1$ (we can similarly handle the case when $x_0$ is assumed to be between $x_1,x_2$). Up to increasing $K$ again, we have points $x,y\in hA$ lying, respectively, between $x_0$ and $x_2$ and between $x_2$ and $x_1$ such that $d(x,y_1),d(y,y_1)$ is bounded in terms of $C$. If $K$ is large enough compared to a constant $\mu$ so that $A$ and hence $hA$ is $(\mu,\mu)$-quasi-isometric to $\R$ this gives a contradiction, see Figure \ref{coarsemon}.
%
%
%
%
%
\end{proof}

\begin{figure}[ht]
\hspace{-.5cm}
\begin{minipage}[bl]{.5\textwidth}
\includegraphics[width=\textwidth]{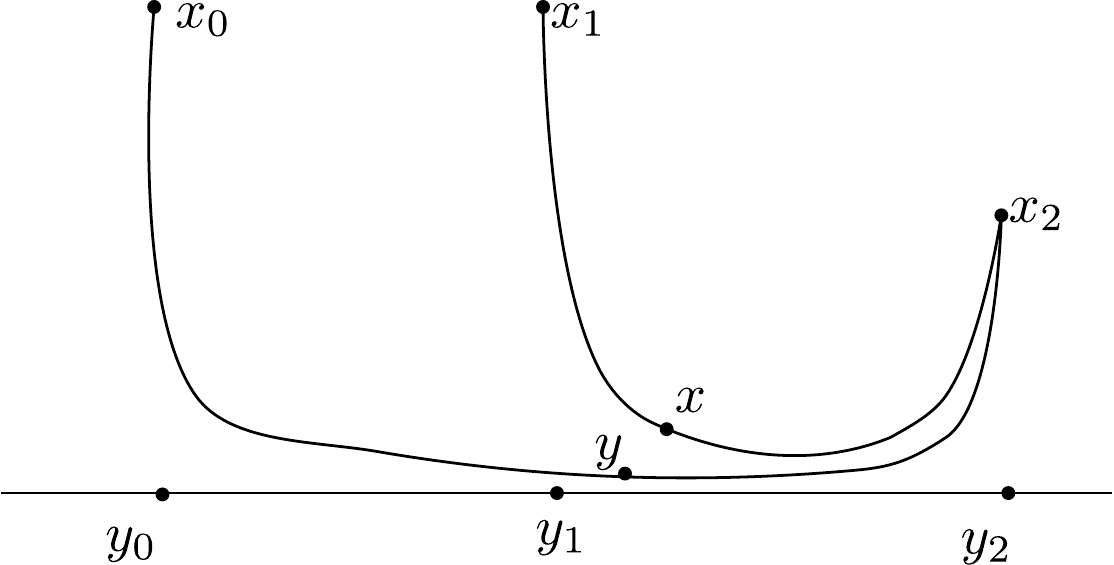}
\caption{}\label{coarsemon}
\end{minipage}
\begin{minipage}[br]{.5\textwidth}
\includegraphics[width=\textwidth]{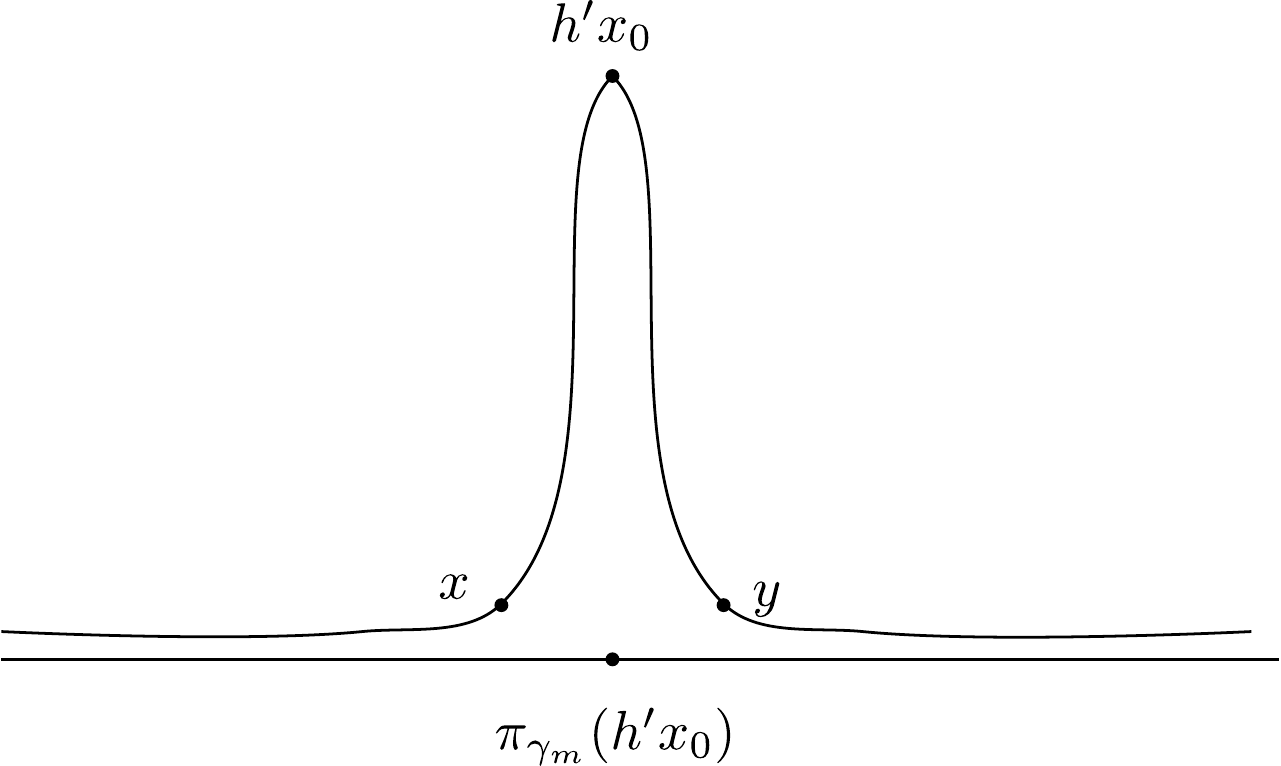}
\caption{}\label{notqgeod}
\end{minipage}
\end{figure}

Let $(X,\calPS)$ be a weak path system on the group $G$, and let $x_0\in X$.

\begin{lemma}\label{boundorcobound}
 Suppose that $g\in G$ is weakly contracting with axis $A$.
Then there exists $K$ such that for each $h\in G$ either $\pi_{A}(hA)$ has diameter bounded by $K$ or it is $K$-dense in $A$. Moreover, there exists $K'$ so that if $d(hx_0,A)\geq K'$ then the first case holds.
\end{lemma}

\begin{proof}
 If $A$ is contained in a neighbourhood of $hA$ of finite radius, then it is easily seen that the second case holds. Otherwise, we can find $h'x_0\in hA$ such that $h'x_0$ is as far as we wish from $A$. We claim that if $d(h'x_0,A)$ is large enough, then $\pi_{A}(h'x_0)$ cannot be between points $p,q$ in $\pi_{A}(h'A)$ that are arbitrarily far from it.

 In fact, by the previous lemma $h'x_0$ would be between points projecting to $p,q$, and hence one can find $x,y\in h'A$ with $h'x_0$ between $x,y$ such that $d(x,\pi_{A}(h'x_0))\leq C$ and $d(y,\pi_{A}(h'x_0))\leq C$, where $C$ is the projection constant. This is easily seen to contradict $h'A$ being a quasi-geodesic, see Figure \ref{notqgeod}.
\par
Now, if $diam(\pi_{A}(hH_0))$ is large enough we can find (see Figure \ref{farcosets2}) a large family $\{N_i\}$ of integers such that
\begin{itemize}
 \item $\pi_{A}(g^{N_i}h'A)$ has positive Hausdorff distance from $\pi_{A}(g^{N_j}h'A)$ whenever $i\neq j$. This gives in particular that $g^{N_i}h'$ and $g^{N_j}h'$ do not belong to the same left coset of $\langle g\rangle$ whenever $i\neq j$.
\item $g^{N_i}h'A$ contains a point in some ball of fixed radius around $g^{m_0}x_0$, as well as a long subpath contained in a neighbourhood of $A$ of radius depending on the projection constant and the quasi-geodesic constants of special paths only. This implies, by Lemma \ref{swpd}, that all left cosets $g^{N_i}h'\langle g\rangle$ contain an element from a certain fixed finite subset of $G$.
\end{itemize}

Therefore there is a bound on the cardinality of $\{N_i\}$, and hence a bound on $diam(\pi_{A}(hA))$.
\end{proof}

\begin{figure}[ht]
\hspace{-.5cm}
\begin{minipage}[bl]{.5\textwidth}
\includegraphics[width=\textwidth]{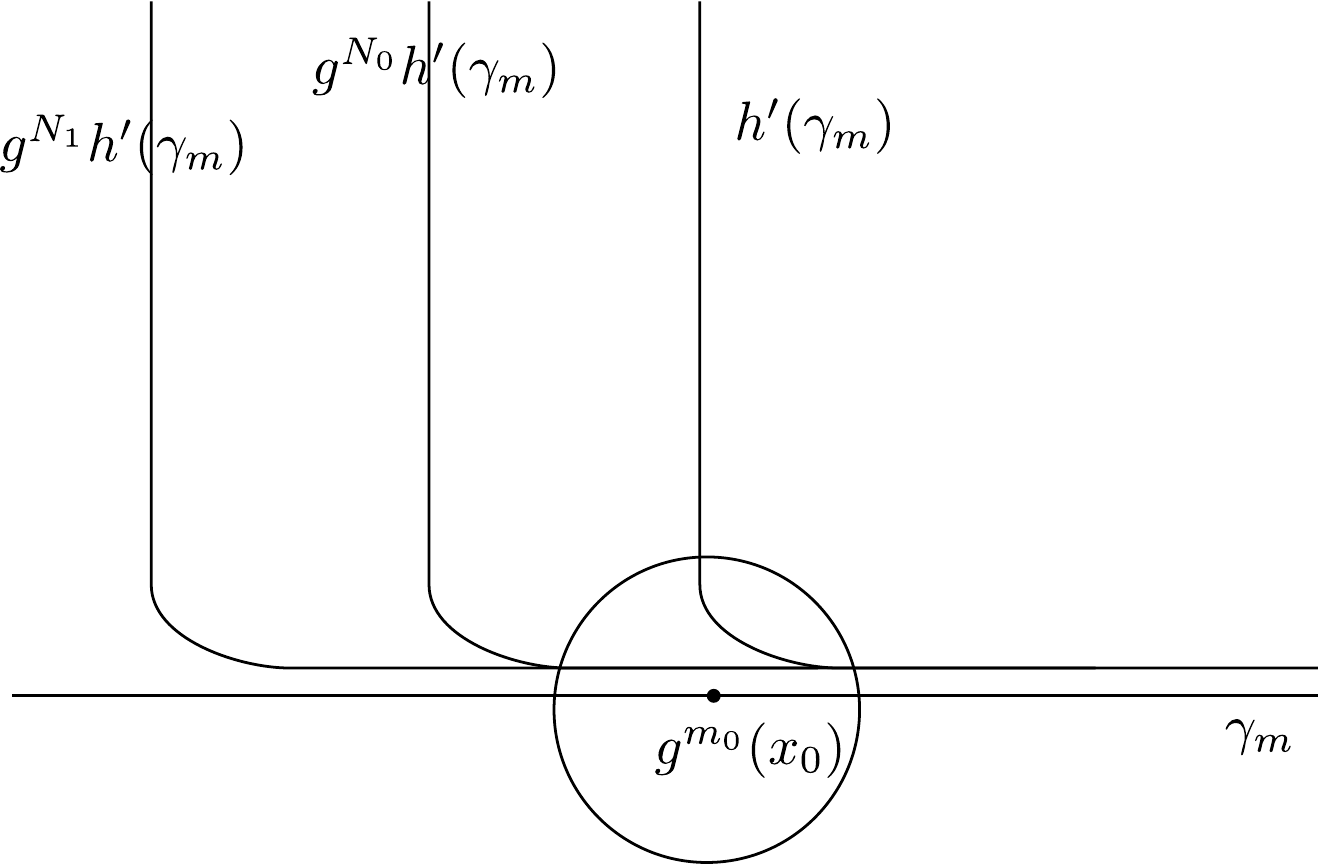}
\caption{}\label{farcosets2}
\end{minipage}
\begin{minipage}[br]{.5\textwidth}
\includegraphics[width=\textwidth]{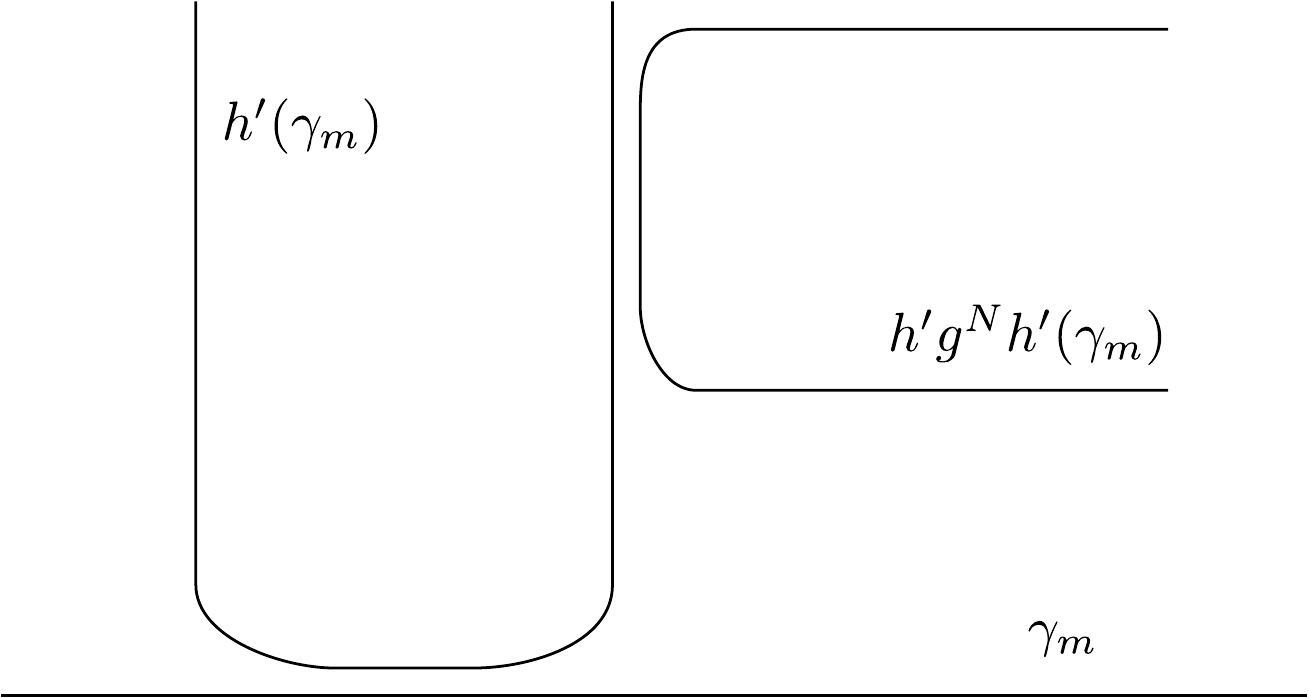}
\caption{}\label{farcosets1}
\end{minipage}
\end{figure}

One can similarly prove the following.

\begin{lemma}\label{farcosets}
 Suppose that $H<G$ is not virtually cyclic and it contains the weakly contracting element $g$ with axis $A$. Then for each $K$ there exists a left coset $h\langle g\rangle\subseteq H$ such that $d(hA, A)\geq K$.
\end{lemma}

\begin{proof}
 As $H$ is not virtually cyclic, it is well-known that it is not quasi-isometric to $\Z$. In particular, we can find $h'$ such that $h'$ is as far as we wish from $\langle g\rangle$ \emph{in a Cayley graph} of $G$. So, either $d(h'x_0,A)$ can be made arbitrarily large, and in this case we have $diam(\pi_{A}(h'A))\leq K$ by the ``moreover`` part of Lemma \ref{boundorcobound}, or we can use Lemma \ref{swpd} to get that $\pi_{A}(h'A)$ is not $K$-dense in $A$, and hence once again $diam(\pi_{A}(h'A))\leq K$ by Lemma \ref{boundorcobound}. It is then easily seen that we can choose $h$ of the form $h'g^Nh'$ with $N$ large enough so that $\pi_{h'A}(A)$ is far from $\pi_{h'A}(hA)$, see Figure \ref{farcosets1}.
\end{proof}

As a consequence of the previous lemmas we get:
\begin{cor}\label{e(g)}
 Each weakly contracting element is contained in a virtually cyclic subgroup, denoted $E(g)$, such that there exists a uniform bound on $\pi_{A}(hE(g))$ for each $h\notin E(g)$.
\end{cor}

\begin{proof}
 Just define $E(g)$ as the collection of all $h$ such that $\pi_{A}(hA)$ has finite Hausdorff distance from $A$. This is clearly a subgroup containing $g$, and it is virtually cyclic by the previous lemma. The uniform bound is a consequence of Lemma \ref{boundorcobound}.
\end{proof}

We point out some extra properties of $E(g)$.

\begin{lemma}\label{normalpow}
 There exists a positive integer $n$ such that for each $h\in E(g)$ we either have $hg^nh^{-1}=g^n$ or $hg^nh^{-1}=g^{-n}$. In particular, for each integer $k$, $\langle g^{nk}\rangle$ is normal in $E(g)$ and $E(g)$ is the centralizer of $g^{2n}$.
\end{lemma}

\begin{proof}
 Choose representatives $h_1,\dots, h_k$ of the left cosets of $\langle g \rangle$ in $E(g)$. Choose $n_i>0$ so that $h_ig^{n_i}h^{-1}_i\in \langle g \rangle$. Such $n_i$ exists because there are finitely many left cosets of $\langle g \rangle$, hence one can find $k_i> m_i$ such that $g^{k_i}h^{-1}_i$ and $g^{m_i}h^{-1}_i$ belong to the same left coset and set $n_i=k_i-m_i$. Choose $n=\prod n_i$. Then it is readily seen that $hg^{n}h^{-1}\in\langle g\rangle$ for each $h\in E(g)$. In particular $hg^{n}h^{-1}=g^k$ for some $k$, and we would like to prove $k=\pm n$. Notice that there exists $j$ such that $h^j\in\langle g \rangle$. Hence
 $$g^{n^j}=h^jg^{n^j}h^{-j}=g^{k^j}.$$
In particular $n^j=k^j$, hence $k=\pm n$, as required.
\end{proof}

\begin{thm}\label{hypemb}
 Let $G$ be any group endowed with a fixed weak path system. For each weakly contracting element $g\in G$, $E(g)$ is hyperbolically embedded in $G$. Viceversa, every infinite order element contained in a virtually cyclic hyperbolically embedded subgroup is weakly contracting for an appropriate action.
\end{thm}

 We will use work of Bestvina, Bromberg and Fujiwara \cite{BBF}, similarly to the proof of \cite[Theorem 4.42]{DGO}. We already described the setting elsewhere, we will now give the relevant definitions again for the convenience of the reader, and state the result we need. Let $\bf{Y}$ be a set and for each $Y\in\bf{Y}$ let $\calC(Y)$ be a geodesic metric space. For each $Y$ let $\pi_Y:{\bf Y}\backslash\{Y\}\to\calP(\calC(Y))$ be a function (where $\calP(Y)$ is the collection of subsets of $Y$). Define
$$d^{\pi}_Y(X,Z)=diam\{\pi_Y(X)\cup\pi_Y(Z)\}.$$
Using the enumeration in \cite{BBF}, consider the following Axioms, for some $\xi\geq 0$:
\par
$(0)$ $diam(\pi_Y(X))\leq \xi$,
\par
$(3)$ there exists $\xi$ so that $\min\{d^{\pi}_Y(X,Z),d^{\pi}_Z(X,Y)\}\leq \xi$,
\par
$(4)$ there exists $\xi$ so that $\{Y:d^\pi_Y(X,Z)\geq \xi\}$ is a finite set for each $X,Z\in\bf{Y}$.
\par
For suitably chosen constants $L,K$, let $\calC(\bf{Y})$ be the path metric space consisting of the union of all $\calC(Y)$ and edges connecting all points in $\pi_X(Z)$ to all points in $\pi_Z(X)$ whenever $X,Z$ are connected by an edge in a certain complex $\calP_K(\bf{Y})$ whose definition we do not need.
\par
We are ready to state the (special case of the) result we will use.

\begin{thm}\cite[Theorem 3.10, Lemma 3.1, Corollary 3.12]{BBF}\label{bbf}
 If $\bf{Y}$ and $d^\pi$ satisfy axioms $(0), (3)$ and $(4)$ and each $\calC(Y)$ is $(\lambda,\mu)$-quasi-isometric to $\R$ for some $\lambda,\mu$ not depending on $Y$, then $\calC(\bf{Y})$ is a quasi-tree (in particular, it is hyperbolic).
\par
Moreover, each $\calC(Y)$ is isometrically embedded in $\calC(\bf{Y})$ and for each $K$ there exists $R$ so that
$$diam(N_K(\calC(X))\cap N_K(\calC(Y)))\leq R$$
whenever $X\neq Y$ are elements of $\bf{Y}$.
\end{thm}

We will use the following characterization of being hyperbolically embedded.

\begin{thm}\cite[Theorem 4.42]{DGO},\cite{Si-metrrh}\label{char}
 Let $G$ be group and $H<G$ a subgroup. Then $H$ is is hyperbolically embedded in $G$ if and only if it acts coboundedly on a space $X$ which is hyperbolic relative to the orbits of the cosets of $H$ (with respect to some basepoint) and the action restricted to $H$ is proper.
%
%
\end{thm}

\emph{Proof of Theorem \ref{hypemb}.}
For technical reasons, in this proof we will allow projections to take values in bounded subsets of the target space. More precisely, we consider a slightly generalized definition of the set $A$ being $\calPS$-contracting with constant $C$ where the projection map $\pi_A$ is allowed to take value in subsets of $A$ of diameter bounded by $C$, while properties $1)$ and $2)$ in Definition \ref{maindef} are left unchanged. If we have a map $\pi_A$ with these properties we can define a projection $\pi'_A$ in the sense of Definition \ref{maindef} just by choosing some $\pi'_A(x)\in\pi_A(x)$ for each $x\in X$. In particular our results, including Lemma \ref{dichot} and Corollary \ref{e(g)}, hold for this more general notion (possibly with different constants).
\par
Let $\bf{Y}$ be the collection of all left cosets of $E(g)$ in $G$ and for each $Y\in\bf{Y}$ let $\calC(Y)$ be a copy of $E(g)$ (more precisely, a copy of the Cayley graph of $E(g)$ with respect to a given finite set of generators), which we regard as identified with $E(g)x_0$, for some given $x_0$. Each orbit $hE(g)x_0$ is a contracting set, and the constant can be chosen uniformly. Actually, we have more. In fact, there exists a collection of equivariant projections on the orbits: choose a projection $\pi'_{e}$ on $E(g)x_0$, define $\pi_e(x)=\bigcup_{h\in E(g)} h\pi'_e(h^{-1} x)$ and define a projection on $hE(g)x_0$ by $\pi_h(x)=h\pi_e(h^{-1}x)$. In order to show that $\pi_h$ is actually a projection we just need to prove that we can uniformly bound $diam(\pi_h(x))$ for each $x$. This follows directly from the easily checked fact that there exist constants $D_1,D_2$ so that for each special path $\g$ from $x$ to $\pi_h(x)$ we have $diam(\g\cap N_{D_1}(E(g)x_0))\leq D_2$.
We can now define
$$\pi_{hE(g)}(h'E(g))=h^{-1}\pi_h(h'E(g)x_0).$$
\par
Axiom $(0)$ follows from Corollary \ref{e(g)}, while Lemma \ref{dichot} implies Axiom $(3)$. Axiom $(4)$ is easy: if $\xi$ is large enough then
$$|\{hE(g):d_{hE(g)}(h_1E(g),h_2E(g))\geq\xi\}|\leq d(h_1E(g)x_0,h_2 E(g)x_0)$$
because all special paths from $h_1E(g)x_0$ to $h_2E(g)x_0$ have long disjoint subpaths each contained in some $hE(g)x_0$ for $hE(g)$ satisfying $d_{hE(g)}(h_1E(g),h_2E(g))\geq\xi$. As all axioms are satisfied, Theorem \ref{bbf} applies. Notice that all constructions are equivariant (including that of $\calP_K(\bf{Y})$), so that there is a natural action of $G$ on $\calC(\bf{Y})$. We wish to show that this action satisfies all properties of Theorem \ref{char}, with $H=E(g)$ and $X=\calC(\bf{Y})$. In fact, in view of Theorem \ref{bbf}, $\calC(\bf{Y})$ is hyperbolic, so that in order to check relative hyperbolicity we have to check that the orbits of the cosets of $E(g)$ are quasi-convex and geometrically separated by Bowditch's characterization of relatively hyperbolic structures on hyperbolic spaces \cite[Section 7]{Bow-99-rel-hyp}. There exists an orbit of $E(g)$ which is an isometrically embedded copy of (the vertices in a Cayley graph of) $E(g)$ and $E(g)$ acts on it in the natural way, so that the orbits of $E(g)$ are quasi-convex and $E(g)$ acts properly. Finally, the orbits of the cosets of $E(g)$ are geometrically separated by the ``moreover'' part of Theorem \ref{bbf}.

The ``vice versa'' is easily checked directly form the definitions, using $X$ as in the definition of hyperbolically embedded subgroups and defining $\calPS$ to be the collection of all geodesics in $X$ (once again, one has to take into account the characterization of the peripheral structures of a relatively hyperbolic space).
\qed

\section{Construction of contracting elements}
\label{constr:sec}

We start with a lemma about concatenations of special paths. More specifically, we consider a concatenation of special paths so that every other path is contracting (more precisely, contained in a contracting set) and show that the resulting concatenation is contracting under suitable conditions. The most important such condition is that the projections of the previous and next contracting paths are close to the corresponding endpoints compared to the length of the contracting path under consideration, see Figure \ref{metrconcat}.

\begin{lemma}
\label{metriconcat}
 Let $X$ be a metric space endowed with a path system.
 For each $C$ there exist $D$ with the following property. Suppose that the sequence of points $\{x_i,y_i\}$ satisfies the following properties for each $i$.
\begin{enumerate}[(a)]
 \item There exists a $C$-contracting set $A_i$ containing $x_i,y_i$,
 \item $diam(\pi_{A_i}(A_{i-1})),diam(\pi_{A_i}(A_{i+1}))\leq C$.
 \item $d(x_i,y_i)\geq d(x_i,\pi_{A_i}(A_{i-1}))+d(y_i,\pi_{A_{i}}(A_{i+1}))+D.$
 \item $\sup\{d(x_i,y_i),d(y_i,x_{i+1})\}<+\infty$.
\end{enumerate}
Then $A:i\mapsto x_i$ (for $i\in \mathbb{Z}$) is a contracting bi-infinite quasi-geodesic. Moreover, any special path from from $x_i$ to $x_j$ has a subpath with endpoints $C$-close to $A_k$ and at least $D/2$-far away from each other whenever $i<k<j$. 
\end{lemma}

\begin{figure}[h]
\centering
 \includegraphics[scale=0.8]{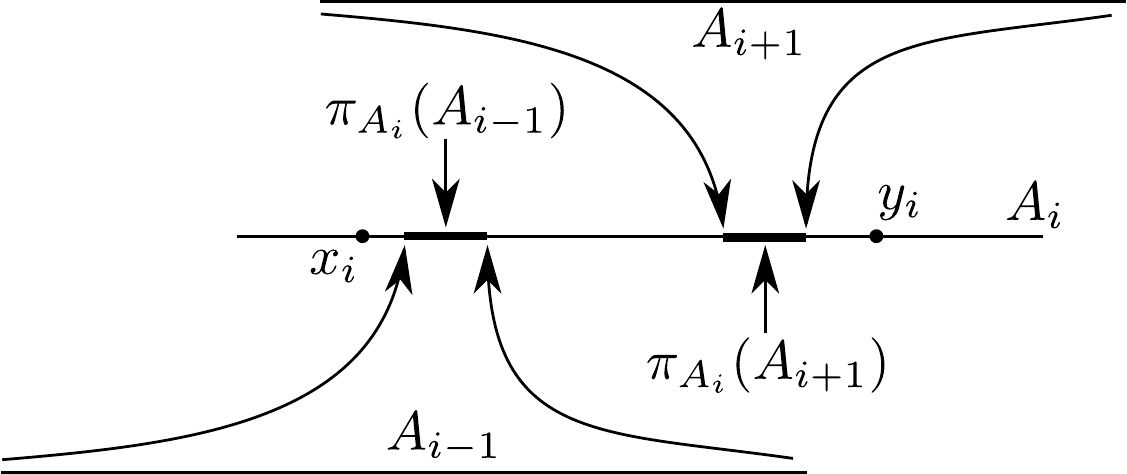}
\caption{The hypotheses of Lemma \ref{metriconcat}}\label{metrconcat}
\end{figure}

At the cost of making the last part of the proof more involved, it is possible to generalise the lemma to finite or one-sided infinite concatenations and substituting (d) with more general conditions (as well as obtaining a more explicit estimate of the contracting constants). We will not need this.

\begin{proof}
Denote $\pi_i=\pi_{A_i}$. We now show that everything coming after $A_i$ in the concatenation projects close to $y_i$, and similarly in the other direction. Formally, setting $d_i=d(x_i,\pi_{i}(A_{i-1})), e_i=d(y_{i},\pi_i(A_{i+1}))$, we would like to show that, for $D$ large enough depending on $C$, we have
\begin{equation}
\pi_i(A_j)\subseteq B(y_i,e_i+C'+C), \pi_j(A_i)\subseteq B(x_j,d_j+C'+C)
\end{equation}
for every $i<j$ and $C'$ as in Lemma \ref{dichot}.
We will show this inductively on $j-i$, the case $j=i+1$ being an easy consequence of (b). For notational convenience, we will show the first containment, the other one can be shown using a symmetric argument.
Suppose that (1) holds whenever $j\leq i+n$ and set $j=i+n+1$ instead. Using the inductive hypothesis we have
$$d(\pi_{i+1}(A_i),\pi_{i+1}(A_j))\geq d(x_{i+1},y_{i+1})-d_{i+1}-e_{i+1}-C'-2C \geq D-C'-2C.$$
For $D$ large enough we can apply Lemma \ref{dichot} and get for each $x\in A_j$
$$d(\pi_{i}(A_{i+1}),\pi_i(x))\leq C',$$
and hence
$$d(y_i,\pi_i(x))\leq d(y_i,\pi_i(A_{i+1}))+diam(\pi_i(A_{i+1}))+C'\leq e_i+C'+C,$$
as required.

The fact that $A$ is a quasi-geodesic now readily follows from the fact that any special path from $x_i$ to $x_j$ has to pass $(d_k+C'+2C)$-close to $x_k$ and $(e_k+C'+2C)$-close to $y_k$ for each $i\leq k \leq j$, if we assume that $D$ satisfies $D-2(C'+2C)\geq C$.

We are only left to show that $A$ is contracting. The idea for defining a contraction is the following. Given $x$, there will be a critical ``time'' $i(x)$ so that the projection of $x$ onto $A_i$ switches from being close to $y_i$ to being close to $x_i$. We define $\pi(x)=x_{i(x)}$. We then notice that, given $x,y$, the projections of $x,y$ are far from each other at times intermediate between the critical ones for $x$ and $y$, and it is not hard to conclude that the projection property holds using this fact.

Let us make this precise. For $x\in X$, define
$$i(x)=\min\{i\in I: d(\pi_i(x),y_i)> e_i+2C'+C\},$$
and set $\pi(x)=x_{i(x)}$. Notice that the set on the right-hand side is non-empty because a special path from $x\in X$ to $A_0$ cannot be $C$-close to $A_i$ for $i$ large enough. In particular, $\pi_{i}(A_0)$ coarsely coincides with $\pi_i(x)$ for such $i$, and $\pi_i(A_0)$ is far from $y_i$. With a similar argument one sees that such set has a lower bound.

The fact that $\pi$ is coarsely the identity map on $\{x_i\}$ easily follows from (1), which implies that $i(x_i)=i$. The fact that it is coarsely the identity on $A$ follows from this and the contraction property shown below.

Consider $x,y\in X$ so that $d(\pi(x),\pi(y))$ is sufficiently large. We have to show that any special path from $x$ to $y$ passes close to $\pi(x),\pi(y)$. Up to exchanging $x$ and $y$, we have $i(y)\geq i(x)+2$. For $i(x)<i<i(y)$, in view of (1) we have
$$d(\pi_{i(x)}(A_i),\pi_{i(x)}(x))\geq d(y_i,\pi_{i(x)}(x))-(e_i+C'+C)>C',$$
so that keeping (c) and Lemma \ref{dichot} into account we have
$$d(x_i,\pi_{i}(x))\leq d(\pi_i(A_{i(x)}),\pi_i(x))+(d_i+C'+C)\leq d_i+2C'+C.$$
On the other hand,
$$d(\pi_{i}(y),y_i)\leq e_i+2C'+C$$
as $i<i(y)$.
For $D$ sufficiently large this implies that any given special path $\gamma$ from $x$ to $y$ passes $(d_i+C'+C)$-close to $\pi_i(x)$ and $(d_i+C'+C)$-close to $\pi_i(y)$. Recall that $i$ was any integer strictly between $i(x)$ and $i(y)$. Let us now choose $i=i(x)+1$. We have
$$d(\gamma,\pi(x))\leq d(\gamma,\pi_i(x))+d(\pi_i(x),x_i)+d(x_i,y_{i(x)})+d(y_{i(x)},x_{i(x)}).$$
We saw already that the first and second term on the right hand side can be uniformly bounded. The remaining two terms are uniformly bounded in view of (d).

Finally, a similar argument with $i=i(y)-1$ provides a bound on $d(\gamma,\pi(y))$.
\end{proof}

The following result is a consequence of Lemma \ref{metriconcat} and gives us a way of constructing new contracting elements from a given one.

\begin{lemma}
\label{hg}
 Let $(X,\calPS)$ be a weak path system on the group $G$. For each $C$ there exists $D$ with the following property. Let $A$ be a $C$-contracting axis for the weakly contracting element $g\in G$, containing $p\in A$. Suppose that $h\in G\backslash E(g)$ is so that
$$d(p,gp)\geq d(\pi_A(hp),p)+d(p,\pi_A(h^{-1}p))+D.$$
Then $hg$, and hence any element conjugate to it, is weakly contracting.
\end{lemma}

It will be convenient to have a slight variation of this lemma, that can be proven in the same way (we leave the details to the reader).

\begin{lemma}
\label{hghg}
 Let $(X,\calPS)$ be a weak path system on the group $G$. For each $C$ there exists $D$ with the following property. Let $A$ be a $C$-contracting axis for the weakly contracting element $g\in G$, containing $p\in A$. Suppose that $h_1,h_2\in G\backslash E(g)$ and $n_1,b_2\in\Z$ are so that
$$d(p,g^{n_1}p)\geq d(\pi_A(h^{-1}_1p),p)+d(p,\pi_A(h_2p))+D,$$
$$d(p,g^{n_2}p)\geq d(\pi_A(h_1p),p)+d(p,\pi_A(h^{-1}_2p))+D,$$
Then $h_1g^{n_1}h_2g^{n_2}$, and hence any element conjugate to it, is weakly contracting.
\end{lemma}

\emph{Proof of Lemma \ref{hg}.}
 Let $\pi_A$ be a $C$-contraction on $A$.
We apply Lemma \ref{metriconcat} with $x_i=(hg)^ihp$, $y_i=(hg)^{i+1}p$ and $A_i=(hg)^ihA$. We claim that all conditions in Lemma \ref{metriconcat} are satisfied, up to suitably increasing the constants. Denote $h_i=(hg)^ih$ for convenience. Set $K=\sup\{b\notin E(g):diam(\pi_A(bA))\}<+\infty$. Condition (d) clearly holds as $d(x_i,y_i)=d(p,gp)$ and $d(y_i,x_{i+1})=d(p,hp)$. Condition (a) is clear setting $\pi_{A_i}(x)=h_i\pi_A(h_i^{-1}x)$. Condition (b) holds for with constant $K$ instead of $C$. In order to show (c), first notice that there is a uniform bound on $d(\pi_A(gx),g\pi_A(x))$ for $x\in X$, as $\pi_A(y)$ coarsely coincides with the first point on any special path from $y$ to $A$ in a suitable neighbourhood of $A$. We can then make the computation:
$$\pi_{A_i}(A_{i+1})=h_i \pi_A(ghA)\subseteq h_i B_{K}(\pi_A(ghp))\subseteq (hg)^{i+1} B_{K'}(\pi_A(hp)),$$
and similarly we get $\pi_{A_i}(A_{i-1})\subseteq h_i B_{K'}(\pi_A(h^{-1}p))$. So, using that $G$ acts by isometries:
$$d(x_i,\pi_{A_i}(A_{i-1}))+d(y_i,\pi_{A_i}(A_{i+1}))\leq $$
$$d(p,\pi_A(h^{-1}p))+d(p,\pi_A(hp))+2K'\leq d(p,gp)-D+2K'\leq d(x_i,y_i)-D+2K'.$$

Now, we are only left to check the WPD property. For a suitable $B$, there are finitely many representatives $k_1,\dots, k_n$ of left cosets of $E(g)$ so that a special path $\alpha$ from $x_0$ to some $(hg)^ix_0$ intersects $N_C(x A)$ in a set of diameter at least $B$ (henceforth, fellow-travels for a long time) only if $x=(hg)^jk_i$ for some $j,i$.

To see this, we can first reduce to showing the analogous property for special paths $\alpha$ connecting a point $C$-close to $x_0$ to a point $C$-close to $hgx_0$ using the ``moreover'' part of Lemma \ref{metriconcat} and using the action of the appropriate $(hg)^jx_0$. Such special path will fellow-travel a translate of $A$ only if the projections of $x_0$, $hgx_0$ on such translate are far enough, and there are only finitely many such translates $xA$ as any special path connecting $x_0$ to $hgx_0$ contains long disjoint subpaths each fellow-travelling some $xA$.

Suppose that $k\in G$ is so that $d(kx_0,x_0),d(k(hg)^Nx_0,(hg)^Nx_0)\leq K$, for $N$ large compared to $K$, and let $\gamma$ be a special path from $x_0$ to $(hg)^Nx_0$. Then both $k\gamma$ and $\gamma$ fellow-travel some $A_j$, and $|j|$ bounded in terms of $K$ if we pick the smallest such $j$. This is because the projection of the endpoints of $k\gamma$ on $A_j$ are close to those of the endpoints of $\gamma$ whenever $A_j$ has distance at least $K+C$ from the endpoints of $\gamma$. 

Hence, $k$ has the form $(hg)^j E(g)((hg)^{j'}k_i)^{-1}$, because $kA'$ only fellow-travels for a long time translates of $A$ of the form $k(hg)^{j'}k_iA$. Also, $|j'|$ can again be bounded in terms of $K$.

\begin{figure}[h]
\centering
 \includegraphics[scale=0.8]{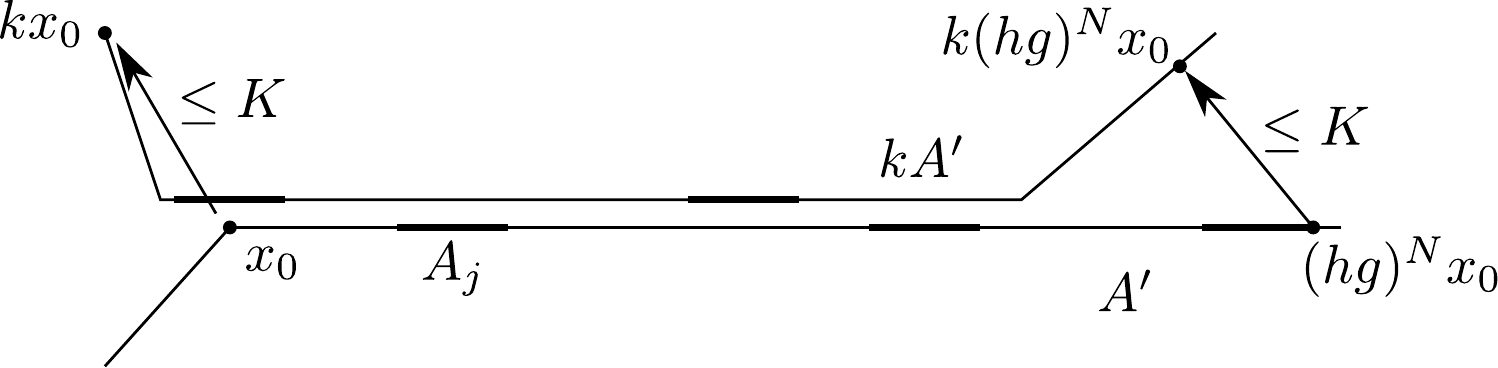}
\caption{$A'$, $kA'$ both intersect far away $K$-balls and fellow-travel in between.}
\end{figure}

Given any $a,b$, there are only finitely elements $g'$ of $E(g)$ so that $d(ag'b,x_0)\leq K$, so that we showed that $k$ lies in one of finitely many double cosets of $E(g)$ and in each of them only finitely many elements satisfy a property satisfied by $k$. Hence, there are finitely many choices of $k$ and we are done.\qedhere

\section{Random walks}
\label{mainproof}

Recall that our aim is to prove the following.

\begin{thm}\label{randomwalks}
  Let $G$ be a non-elementary group supporting the path system (resp. weak path system) $(X,\calPS)$ and containing a contracting (resp. weakly contracting) element. Then the probability that a simple random walk $\{X_n\}$ supported on $G$ gives rise to a non-contracting (resp. non-weakly-contracting) element decays exponentially in the length of the random walk.
\end{thm}

The reader may wish, on first reading, to take the short-cut suggested in Subsection \ref{shortcut}, where we indicate how to prove the slightly weaker result that the probability of ending up in a non-contracting element goes to $0$ without the extra information that it does so exponentially fast.

\subsection{Notation}
We fix the notation of the theorem throughout this section. Additionally, we let $W_n(S)$ be the set of words of length $n$ in the generating system of $G$ that defines the random walk $\{X_n\}$. For $w$ a word in $S$, we denote by $g(w)$ the corresponding element of $g$. Finally, $p$ is a fixed basepoint of $X$.
To simplify the notation, for any contracting set $A$,
$$d_A(\cdot,\cdot)=d(\pi_{A}(\cdot),\pi_{A}(\cdot)).$$

\subsection{Many contracting subwords}
In order to keep the proof as elementary and self-contained as possible, rather than using more refined estimates on sums of independent random variables we will only use the following well-known inequality: for each positive integers $n\geq m$, we have
 
 \begin{equation}
 \label{binom}
  \binom{n}{m}\leq \left(\frac{ne}{m}\right)^m.
 \end{equation}

Let $w_0\in W_m(S)$ be a word of length $m$ (we will be interested in a word representing a weakly contracting element).
For each $k$ denote $W_{w_0}^k$ the set of words obtained concatenating $k$ words of the form $w_0$ or $w^{-1}_0$. For a word $w$, denote by $w[j,k]$ the subword starting with the $j$-th letter and ending with the $(k-1)$-th letter. Let $l_k=km$ and denote $\calI_{w_0}^k(w)=\{i\in 2\N:w[il_k,(i+1)l_k]\in W_{w_0}^k\}$. In words, we split $w$ into subwords of length $l_k$, keep the even ones and count how many of those are concatenations of $w_0$ and $w_0^{-1}$.

The following lemma tells us that we expect the size of $\calI_{w_0}^k(w)$ to be linear in the length $n$.

\begin{lemma}
\label{manywpd}
For each $w_0\in W_m(S)$ and $k\geq 1$ there exists $C_0\geq 1$ so that
$$\matP[\#\calI_{w_0}^k(X_n)\leq n/C_0]\leq C_0e^{-n/C_0}$$
for each $n\geq 1$.
\end{lemma}

\begin{proof}
For notational convenience, we drop the ``$w_0$'' subscripts. Let $K$ be a large enough constant to be determined by the following argument.
Let $w$ be a word of length $n$ and set $n_0=\lfloor n/(2l_k)\rfloor$. We can assume that $n$ is large enough that $\lfloor n/K\rfloor\geq n_0/(2K)$.

If $\#\calI^k(w)\leq n_0/K$ then $\calI^k(w)$ is contained in a set $I\subseteq \{0,2,\dots,2n_0\}$ of cardinality $\lfloor n_0/K\rfloor$, and for each $i\in \{0,\dots,2n_0\}\backslash I$ we have that $w[il_k,(i+1)l_k]$ (is defined and) does not belong to $W^k$. The probability of a word of length $l_k$ not belonging to $W^k$ is some $\rho<1$ (depending on $k$ and the length of $w_0$).
We now use inequality \ref{binom}. Summing over all possible sets $I$ we can bound the probability in the statement by
$$\binom{n_0}{\lfloor n_0/K\rfloor}\rho^{n_0-\lfloor n_0/K\rfloor}\leq (2Ke)^{n_0/K}\rho^{(1-1/K)n_0}=\big((2Ke)^{1/K}\rho^{1-1/K}\big)^{n_0}.$$
If $K$ is large enough, depending on $\rho$ only, we have $(2Ke)^{1/K}\rho^{1-1/K}<1$ (because the first term tends to $1$ and the second one tends to $\rho$ for $K\to\infty$). This gives us the exponential decay we were looking for.
\end{proof}

\subsection{Avoiding $E(g)$}

\begin{lemma}
\label{avoideg}
For any weakly contracting element $g$ of $G$ there exists $C_1\geq 1$ so that
 $$\matP[X_n\in E(g)]\leq C_1e^{-n/C_1}.$$
\end{lemma}

\begin{proof}
As we point out in the Appendix, since $H$ is non-elementary, it must contain a free group on two generators and so it is not amenable. Hence, by a result of Kesten, see e.g. \cite[Corollary 12.12]{Wo}, there exists $K\geq 1$ so that for each $g\in G$ and $n\geq 1$ we have 
 $$\matP[X_n=g]\leq Ke^{-n/K}.$$
Also, up to increasing $K$, in the ball of radius $n$ \emph{in the Cayley graph} of $G$ around $1$ there are at most $Kn$ elements of $E(g)$, because the inclusion of $E(g)$ in $G$ is a quasi-isometric embedding (if it was not, the orbit of $g$ could not be a quasi-geodesic in $X$). Hence,
$$\matP[X_n\in E(g)]\leq (Kn)(Ke^{-n/K}),$$
which decays exponentially.
\end{proof}

\subsection{One small projection}
In this section we show that we expect the projection of a random point on an axis of a fixed weakly contracting element to be close to $p$, the basepoint of $X$. This is the key fact we need to use Lemma \ref{hg} or \ref{hghg}, our methods for constructing weakly contracting elements. We use the standard notation $\matP[\cdot|\cdot]$ for conditional probabilities.

\begin{lemma}
\label{rndprojsmall}
Let $w_0$ be a word of even length representing a weakly contracting element $g\in G$, with axis $A$. 
For every $k$ large enough depending on $w_0$ the following holds. There exists $C_2\geq 1$ so that for each $j\geq 1$, $t\geq 0$ and $I\subseteq 2\N$ we have 
$$\matP\left[d(p,\pi_A(X_np))\geq t|\calI^k_{w_0^j}(X_n)=I\right]\leq C_2e^{-t/C_2}.$$
\end{lemma}

The reason why we require $w_0$ to be even is just to make the existence of certain words slightly easier to prove. The reason why we consider a fixed $\calI^k_{w_0^j}$ is merely technical: The final computation in the proof of the main theorem requires to consider a conditional probability, and this propagates backwards up to here.

We actually do not need the exponential decay, just a convergence to $0$, but this is what the proof gives anyway.

\begin{proof}
We want to find $D>0$ with the following property. Let $w$ be a word with $\calI^k_{w_0^j}(w)=I$ and $d(p,\pi_A(g(w)p))\geq (i+1)D$ for some integer $i\geq 2$. We want to associate to $w$ a new word $w'$, again with $\calI^k_{w_0^j}(w')=I$ and of the same length as $w$, so that $d(p,\pi_A(g(w')p))\in [iD, (i+1)D)$. Also, we will make sure that at most $N$ words $w$ are mapped to the same $w'$, where $N$ does not depend on $j$. This is enough for our purposes. In fact, once we manage to do so we have, setting $C_n=``\calI^k_{w_0^j}(X_n)=I''$,
$$\matP\left[d(p,\pi_A(X_np))\geq (i+1)D\; |\; C_n\right]\leq$$
$$N \left(\matP\left[d(p,\pi_A(X_np))\geq iD \;|\; C_n\right]-\matP\left[d(p,\pi_A(X_np))\geq (i+1)D\;|\; C_n\right]\right),$$
and hence
$$\matP\left[d(p,\pi_A(X_np))\geq (i+1)D\; |\; C_n\right]\leq (1+1/N)^{-1}\matP\left[d(p,\pi_A(X_np))\geq iD\; |\; C_n\right].$$
A straightforward inductive argument completes the proof.

We are left to show how to construct $w'$ with the required properties. Fix some word $s$ so that $g(s)\notin E(g)$ and let $C'$ be as in Lemma\ref{dichot}. For reasons that will be clearer later, we need $k$ large enough with the property that there exist words $v_1,v_2,v_3,v_4$ of length $l_k$ (recall that $l_k=km$ for $m$ the length of $w_0$) so that
\begin{enumerate}
\item $g(v_1),g(v_2)\in h_1 E(g)h_1^{-1}$, $g(v_3),g(v_4)\in h_2 E(g)h_2^{-1}$,
\item $h_1 E(g)\neq h_2 E(g)$,
\item $d(g(v_2)g(v_1)^{-1}x,x)$, $d(g(v_4)g(v_3)^{-1}y,y)$ are larger than a constant determined by the argument below for any $x\in h_1A,y\in h_2A$.
\end{enumerate}

It is easy to see that such words exist. For example, one can start choosing any $h_1,h_2$ as required (recall that $G$ is non-elementary so that $[G:E(g)]=\infty$) and words $u_1,u_2$ representing them. Then one finds words $t_1,t_2$ both of even length so that $g(t_i)\in E(g)$ and with long translation distance on $A$. One can then set $v'_1=u_1t_1u_1^{-1}$, etc. and prolong such words to get the $v_i$'s by adding final subwords (of even length) that represent the identity, making sure that the prolonged words have the same length and that such length is a multiple of the length of $w_0$. 

We will subdivide $w$ into three subwords $w=w_1w_2w_3$, and substitute the middle subword with one of the $v_i$'s. The length of $w_2$ will be $l_k$. Also, we let $w_1$ be the shortest initial subword of $w$ whose length is an \emph{odd} multiple of $l_k$ and so that $d(p,\pi_A(w))\geq (i+0.5)D$. First of all, it is now readily seen that the property that $w\mapsto w'$ is bounded-to-1 is satisfied. It is also readily seen that $\calI^k_{w_0^j}(w_1v_iw_3)=I$ for each $i$ (recall that $\calI^k_{w_0^j}(\cdot)$ is a set of even numbers).

We now only have to gain control on the projection on $A$. In view of (2) at least one of $g(w_1)h_1, g(w_1)h_2$ is not in $E(g)$, let us assume the former. Translating by $g(w_1)^{-1}$ we get
$$d_{g(w_1)h_1A}(g(w_1v_1w_3)p,g(w_1v_2w_3)p)=d_{h_1A}(g(v_1w_3)p,g(v_2w_3)p),$$
and we would like to show that this quantity is large. Any special path $\gamma$ from $g(v_1w_3)p$ to any point in $h_1A$ enters a suitable neighbourhood of $h_1A$ close to the projection point. The special path $g(v_2)g(v_1)^{-1}\gamma$ connects $g(v_2w_3)p$ to some point in $h_1A$ because $g(v_1),g(v_2)$ stabilize $h_1A$ in view of (1). Also, (3) implies that the entrance points in the neighbourhood of $h_1A$ are far, and hence so are the projection points of $g(v_1w_3)p$, $g(v_2w_3)p$.

The argument above shows that at most one of $g(w_1v_1w_3)p,g(w_1v_2w_3)p$ can project closer than $C'$ to $\pi_{g(w_1)h_1A}(A)$, where $C'$ is as in Lemma \ref{dichot}. In particular, if this holds for, say, $g(w_1v_1w_3)p$, we get a uniform bound on $d_A(g(w_1v_1w_3)p,(g(w_1)h_1p))$ and hence on $d_A(g(w_1v_1w_3)p,g(w_1)p)$. Therefore, for $D$ large enough the inequality $d_A(g(w_1v_1w_3)p,g(w_1)p)<D/2$ holds, and thus $d(\pi_A(g(w_1v_1w_3)p),p)\in [iD,(i+1)D)$, as required.
\end{proof}

\subsection{Many small projections}

Let $\calA=(A,w_0)$ be a pair where $A$ is the axis of the weakly contracting element $g(w_0)$ represented by the word $w_0\in W_m(S)$.

\medskip
\emph{Standing assumptions:} we assume that the length of $w_0$ is even. Also, we fix an odd $k\geq 10^9$ satisfying Lemma \ref{rndprojsmall} and so that $\matP[X_n\in E(g)]\leq 10^{-3}$ for each $n\geq k$ (see Lemma \ref{avoideg}).
\medskip

We denote $\calI_{w}=\calI^k_{w}$ for any word $w$. The reason why we fix an odd $k$ is that any subword counted by $\calI$ corresponds to a non-trivial power of $g(w_0^j)$, a fact that will be used later.
Slightly extending the notation we already used, for $i<j\leq n$ and a word of length $n$ we denote by $w[j,i]$ the concatenation of $w[j,n]$ and $w[n,i]$ ($w$ is hence considered as a cyclic word).
Given a pair as above, a word $w$ and an integer $\hat{i}$ so that $\hat{i}l$ is not larger than the length of $w$, denote
$$\calJ_{\calA,L}^+(w,\hat{i})=\left\{i\neq \hat{i}:d\left(p,\pi_A\Big(g(w[(i+1)l,\hat{i}l])p\Big)\right)\leq L\right\}\cap\calI_{w_0}(w),$$
for $l=km$.

\begin{figure}[h]
\centering
\includegraphics[scale=0.8]{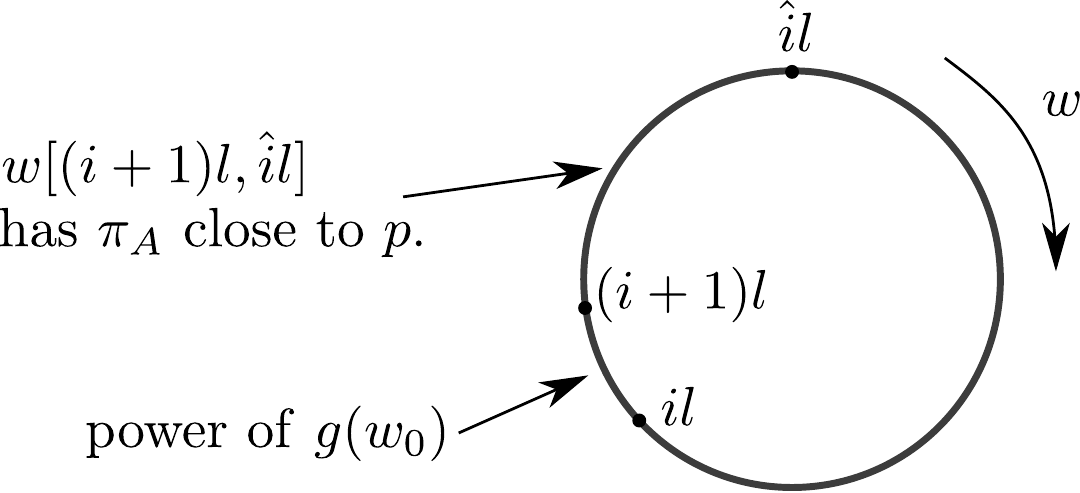}
\caption{$i$ is in $\calJ_{\calA,L}^+(w,\hat{i})$ if...} 
\end{figure}

In words, we replace $w$ by a cyclic permutation and we record the indices so that the projection of a final subword of $w$ following one of the occurrences of powers of $g(w_0)$ counted by $\calI_{w_0}$ has distance at most $L$ from the basepoint $p$. Similarly, reading $w$ backwards we set
$$\calJ_{\calA,L}^-(w,\hat{i})=\left\{i\neq \hat{i}:d\left(p,\pi_A\Big(g(w[(\hat{i}+1)l,il])^{-1}p\Big)\ \ \right)\leq L\right\}\cap\calI_{w_0}(w).$$
\medskip
Suppose that some $j$ is fixed and set $\calA_j=(A,w_0^j)$. Then, given a word $w$, we say that $(\hat{i}_1,\hat{i}_2)\in \left(\calI^k_{w_0^j}(w)\right)^2$ is a \emph{good} $L$-pair if $\hat{i}_1\in \calJ_{\calA_j,L}^+(w,\hat{i}_2)\cap \calJ_{\calA_j,L}^-(w,\hat{i}_2)$ and vice versa, i.e. if all projections suggested by the following picture are ``smaller'' than $L$.

\begin{figure}[h]
\centering
\includegraphics[scale=0.8]{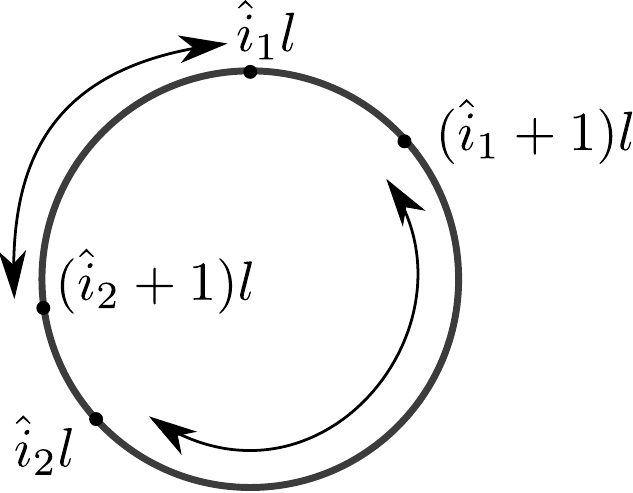}
\caption{A good pair. Arrows correspond to controlled projections} 
\end{figure}

The following proposition tells us that we expect to find a good pair of well-separated indices, once we replace $w_0$ by some large power to increase the translation distance on $A$.



\begin{prop}
\label{jpmneqempty}
There exists $L\geq 0$ with the following property. For each $\calA$ as above and $D\geq0$ there exists $j\geq 1$ so that $\calA_j=(A,w_0^j)$ satisfies the following property for some $C_3\geq 1$. For each $n\geq 1$, $I\subseteq \{0,2,\dots,2\lfloor n/(2lj)\rfloor\}$ of cardinality at least $100C_3$ we have
$$\matP\left[\forall \hat{i}_1,\hat{i}_2\in I,\, |\hat{i}_1-\hat{i}_2|\geq \frac{\#I}{C_3} \Rightarrow (\hat{i}_1,\hat{i}_2) {\mathrm\ is\ not\ } L{\mathrm -good}\; |\; \calI_{w_0^j}(X_n)= I\right]\leq C_3e^{-\#I/C_3}.$$
Moreover, $d(p,g^jp)>2L+D$.
\end{prop}

The proof is (essentially) an induction on $|I|$, the base case being Lemma \ref{rndprojsmall}.

Let us now choose $j$ and $L$. Let $C$ be a contraction constant for $A$. We fix $L>C'+C$ so that $\matP[d(p,\pi_A(X_np))\geq L-C'-C|\calI_{w_0^j}(X_n)=J]\leq 10^{-3}$ for any $J\subseteq 2\N$ and $C'$ as in Lemma \ref{dichot} (Lemma \ref{rndprojsmall} guarantees that $L$ exists). Also, we pick $j$ so that $d(p,g(w_0)^jp)>2L+D$. We now drop all subscripts in $\calI,\calJ^{\pm}$ if they are the ones that can be inferred from the choices we just made.

Proposition \ref{jpmneqempty} will follow easily from the following lemma.

\begin{lemma}
\label{45}
There exists $K$ so that
 $$\matP\left[\exists \hat{i}\in I: \#\calJ^+(X_n,\hat{i})\leq \frac{4}{5}\#I\; |\; \calI(X_n)= I\right]\leq Ke^{-\#I/K},$$
and the same holds for $\calJ^-$.
\end{lemma}

\emph{Proof.}
We prove the lemma for $\calJ^+$, the proof for $\calJ^-$ is symmetric.

We can fix $\hat{i}\in I$ and just show the same conclusion with ``$\exists \hat{i}\in I$'' removed, the lemma will follow upon taking a sum over all $\hat{i}\in I$.
For distinct $i_1,i_2\in I$ we write $i_1<i_2$ if either $i_1< i_2\leq \hat{i}$, $\hat{i}< i_1< i_2$ or $i_2\leq \hat{i}< i_1$ (i.e., we read $w$ in cyclic order starting after $\hat{i}lj$ and determine the order in this way).

{\bf Claim:} There exists $\rho<10^{-2}$ so that for each $i_1<\dots<i_l$ contained in $I$ we have
$$\matP[i_1\notin \calJ^+(X_n,\hat{i}) |\calI(X_n)=I, i_2,\dots,i_l\notin \calJ^+(X_n,\hat{i})]\leq \rho.$$
\smallskip

\emph{Proof of claim.} If the conditional hypothesis holds for the word $w$ and $i_1\notin \calJ^+(w,\hat{i})$, then, as we are about to show, at least one of the conditions below hold.
Denote $w_1=w[(i_1+1)lj,i_2lj]$, $w_2=w[i_2lj,(i_2+1)lj]$, $w_3=w[(i_2+1)lj,\hat{i}lj]$.

\begin{figure}[h]
\centering
\includegraphics[scale=0.8]{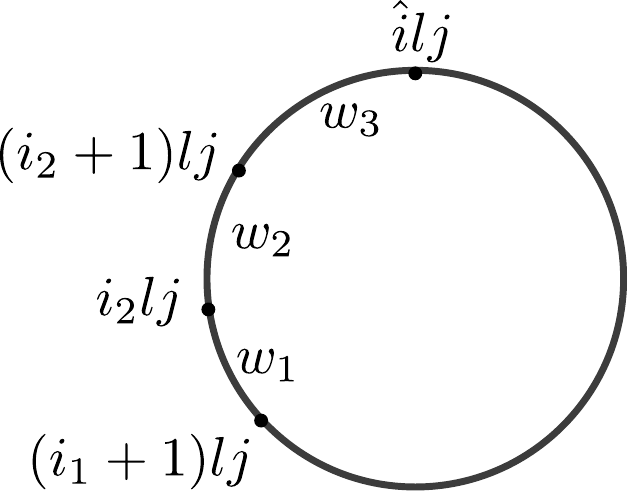} 
\end{figure}

\begin{enumerate}
 \item $d(p,\pi_A(g(w_2w_3)p))\leq L$. ($w_2w_3$ projects close to $p$.)
 \item $d(\pi_{g(w_1)A}(p), g(w_1)p)>L-C'-C$. ($w_1^{-1}$ projects far from $p$.)
 \item $d(p,\pi_A(g(w_1)p))> L-C'-C$. ($w_1$ projects far from $p$.)
 \item $g(w_1)\in E(g)$.
\end{enumerate}

 In fact, suppose that all conditions fail. Then the negation of (1) multiplied on the left by $g(w_1)$ and the negation of (2) imply 
$$d_{g(w_1)A}(p,g(w_1w_2w_3)p)\geq $$
$$d(g(w_1)p,\pi_A(g(w_1w_2w_3)p))-d(\pi_{g(w_1)A}(p), g(w_1)p)> C'+C.$$
Thus, by Lemma \ref{dichot} (and $g(w_1)\notin E(g)$), $\pi_A(g(w_1w_2w_3)p)$ is $C'$-close to $\pi_A(g(w_1)A)$. But $\pi_A(g(w_1)A)$ has diameter at most $C$ and contains a point at distance at most $L-C'-C$ from $p$ by the negation of (3). In particular, $\pi_A(g(w_1w_2w_3)p)$ is $L$-close to $p$.
As $w_1w_2w_3=w[(i_1+1)lj,\hat{i}lj]$, we see that by the very definition this means that $i_1$ is in $\calJ^+(w,\hat{i})$.

We can conclude the proof of the Claim. The point about the conditions stated above is that they are (or, in the case of (1), are easily converted into) properties of $w_1,w_2$, so that they are ``independent enough'' from the conditional hypothesis, which is mostly about $w_3$.

The probability that (2) holds is bounded by $10^{-3}$ by our choice of $L$, and similarly for (3). The probability that (4) holds is also at most $10^{-3}$, by our choice of $k$ (see Standing Assumptions). Given any possible $w_3$, there is at most one power of $g(w_0)^j$ so that $d(\pi_A(g(w_2w_3)p), p)\leq C'+C$. The probability that an element of $W^k$ represents such power of $g(w_0)^j$ is smaller than $10^{-3}$ (as can be seen for example using Inequality \ref{binom}, using $k\geq 10^9$). This observation proves that the probability that $1$ happens is smaller than $10^{-3}$ as well. Summing these probabilities we get the desired bound.\qed

\smallskip
Using the claim, it is not hard to complete the proof with an argument similar to Lemma \ref{manywpd}. For $w\in W_n(S)$ so that $\calI(w)=I$ denote $\hat{I}(w)=I\backslash \calJ^+(w, \hat{i})$. For each $I_0\subseteq I$ containing the elements $i_1<\dots<i_l$ we have
$$\matP[\hat{I}(X_n)\supseteq I_0|\calI(X_n)=I]=$$
$$\Pi_{t\geq1} \matP[i_t\notin \calJ^+(X_n,\hat{i}) |\calI(X_n)=I, i_{t+1},\dots,i_l\notin \calJ^+(X_n,\hat{i})]\leq \rho^l.$$
If the word $w$ is so that $\#\calJ^+(w)\leq \frac{4}{5}\#\calI(w)$ but $\calI(w)=I$, then there exists $I_0\subseteq \hat{I}(w)$ of cardinality $\lceil \#I/5\rceil$. The final, straightforward, computation is just summing over all possible $I_0$'s and using inequality \ref{binom}:
$$\matP\left[\#\calJ^+(X_n)\leq \frac{4}{5}\#\calI(X_n)|\calI(X_n)= I\right]\leq \binom{\# I}{\lceil \#I/5\rceil} \rho^{\lceil\#I/5\rceil}\leq$$
$$(100)^{\#I/5 +1}\rho^{\#I/5}\leq 100 (100\rho)^{\#I/5},$$
and the last term decays exponentially. (We used $(\#I e)/\lceil \#I/5\rceil\leq 100$, which follows from $\#I\geq 100 C_3$.)\qed
\medskip

\emph{Proof of Proposition \ref{jpmneqempty}.} Suppose we fixed a word $w$ of length $n$ so that $\calI(w)=I$. We say that $\hat{i}_1\in I$ is \emph{good} for $\hat{i}_2\in I$ if $\hat{i}_1\in \calJ_{\calA_j,L}^+(w,\hat{i}_2)\cap \calJ_{\calA_j,L}^-(w,\hat{i}_2)$. We want to show that if $w$ satisfies $\#\calJ^{\pm}(w,\hat{i})\leq \frac{4}{5}\#I$ for each $\hat{i}\in I$ then there exists a pair of indices $\hat{i}_1$, $\hat{i}_2$ so that $|\hat{i}_1-\hat{i}_2|\geq \#I/100$ but $\hat{i}_1$ is good for $\hat{i}_2$ and vice versa.
The fraction of words without this property is exponentially small in $\#I$ by Lemma \ref{45}.

Fix instead a word satisfying the property. Then we can make the following counting argument. Given any index $\hat{i}_0$, at least $3\# I/5$ indices are good for $\hat{i}_0$. In particular, there is an index $\hat{i}_1$ which is good for at least $3\# I/5$ indices. Out of these, there are at least $\#I/5$ indices which are good for $\hat{i}_1$. If $\hat{i}_2$ the furthest such index from $\hat{i}_1$, then the pair $\hat{i}_1$, $\hat{i}_2$ is the one we were looking for.

\subsection{Proof of main theorem} We are only left to put the pieces together. Then fix $w_0,j,k,L$ as in Proposition \ref{jpmneqempty} for $D$ as in Lemma \ref{hghg}. An exponentially small fraction $\leq C_4e^{-n/C_4}$ of the words $w$ of any given length is so that $\#\calI_{w_0^j}^k(w)\leq n/C_4$. Let $C_3$ be given by Proposition \ref{jpmneqempty}. Summing Lemma \ref{avoideg} over all subwords of length at least $n/(C_3C_4)$, we see that a fraction $\leq C_5e^{-n/C_5}$ of words contains a subword of length at least $n/(C_3C_4)$ representing an element of $E(g(w_0))$.

Given a word $w$ we say that $(\hat{i}_1,\hat{i}_2)\in \left(\calI^k_{w_0^j}(w)\right)^2$ is a \emph{very good} pair if it is a good pair and $|\hat{i}_1-\hat{i}_2|\geq \#\calI^k_{w_0^j}(w)/C_3$.

Recall that we chose $k$ to be odd so that any subword counted by $\calI^k_{w_0^j}$ corresponds to a non-trivial power of $g(w_0^j)$. Lemma \ref{hghg} is a criterion for elements (conjugate to an element) of the form $h_1g^{n_1}h_2g^{n_2}$ to represent a weakly contracting element. Unravelling the definitions we see that if a word $w$ has no subwords of length at least $n/(C_3C_4)$ representing an element in $E(g(w_0))$ and it admits a very good pair then $w$ has, up to cyclic permutation, the form prescribed by Lemma \ref{hghg} and in particular it represents a weakly contracting element.

Finally, keeping this and the estimates we gave above into account (which we use to get rid of words with few occurrences of power of $g(w_0^j)$ and long subwords in $E(g)$) we can make the following estimate for $n$ large:
$$\matP[g(X_n) {\mathrm\ not\ weakly\ contracting}]\leq C_4e^{-n/C_4}+C_5e^{-n/C_5}+$$
$$\sum_{|I|\geq n/C_1} \matP[\forall \hat{i_1},\hat{i_2}\in I, \, (\hat{i_1},\hat{i_2}) {\ is\ not\ very\ good}|\calI(X_n)= I]\;\matP[\calI(X_n)= I]\leq$$
$$C_6e^{-n/C_6}+C_3e^{-n/C_3}\sum \matP[\calI(X_n)= I]\leq C_7e^{-n/C_7},$$
for appropriate $C_6,C_7$.\qedhere

\subsection{Proof of the corollaries}
We now prove the following corollary of the main theorem.

\begin{cor}
Let $G_1,G_2$ be a non-elementary group supporting the path system (resp. weak path system) $(X,\calPS)$ and containing a contracting (resp. weakly contracting) element.
 \begin{itemize}
  \item For any isomorphism $\phi:G_1\to G_2$ there exists a (weakly) contracting $g\in G$ so that $\phi(g)$ is also (weakly) contracting.
  \item If $H$ is an amenable group and $\phi:G_1\to H$ is any homomorphism, for any $h\in \phi(G_1)$ there exists a (weakly) contracting $g\in \phi^{-1}(h)$.
 \end{itemize}
\end{cor}

\begin{proof}
 The proof of both points relies on considering a simple random walk $\{X_n\}$ on $G_1$ and pushing it forward to a random walk on the target group.

In the first case, one obtains a simple random walk $\{Y_n\}$ on $G_2$. For $n$ large enough, the probability that both $X_n$ and $Y_n$ represent a non-weakly contracting element are less than $1/2$, so that for such $n$ there exists a (weakly) contracting element of $G_1$ of word length at most $n$ that is mapped to a (weakly) contracting element of $G_2$. (A probability-free way of saying this is that, for $S$ a generating system for $G$ and $n$ large enough, the sets of words of length $n$ in both $S$ and $\phi(S)$ that do not represent a (weakly) contracting element have cardinality less than a half of the cardinality of $W_n(S)$.)

In the case of $2)$, the push-forward of $\{X_n\}$ is a symmetric finitely supported random walk on $H$. It is well-known that the probability that such a random walk ends up in a given element $h\in H$ decays slower than exponentially, provided that $h$ is in the group generated by the support of the random walk (which is in our case $\phi(G_1)$), see \cite[Theorem 12.5]{Wo}. Hence, for such $h$ and each large enough $n$, $\matP[X_n \text{\ (weakly)\ contracting}|\phi(X_n)=h]>0$, and we are done.
\end{proof}

\subsection{A simpler proof of a simpler theorem}
\label{shortcut}
There is a shortcut in the proof of genericity of weakly contracting elements if one wants to show that the probability of ending up in a non-weakly contracting element goes to 0, without the extra information that it does so exponentially fast.

This is because there exists $K$ so that, given a word $w$ of length at most $\log(n)/K$, the probability that a random word of length $n$ does not contain $w$ goes to $0$ for $n$ going to infinity. The idea is then to look for powers of a word representing a weakly contracting element and using (the appropriate versions of) Lemmas \ref{avoideg} and \ref{rndprojsmall} to show that the hypothesis of Lemma \ref{hg} are satisfied with high probability. More details are below.

Fix a word $w_0$ representing the weakly contracting element $g$, with axis $A$. We claim that for an appropriate $f(n)\in O(\log(n))$, ``most'' words $w$ of length $n$ (meaning all except for a fraction of size going to $0$ as $n$ goes to infinity) have a uniquely determined decomposition $w=w_1w_2w_3$ with the following properties:
\begin{enumerate}
\item $w_2$ is a power of $w_0$ with exponent $f(n)$,
\item either $w_1=\emptyset$ and $w_3$ is any word of length $n-f(n)$ or $w_3w_1$ is any word of length $n-f(n)$ not ending with $w_0$.
\item $g(w_3w_1)\notin E(g(w_0))$,
\item $d(p,\pi_A(g(w_3w_1)p))\leq f(n)/3$.
\end{enumerate}
 
We already commented on (1), and (2) just says that $w_2$ is the first occurrence of the $w_0^{f(n)}$.

Regard $w$ as a cyclic word. Then the probability that any given subword of $w$ of length $n-f(n)$ represents an element of $E(g(w_0))$ is exponentially small in $n$ by Lemma \ref{avoideg}, so that a stronger version of (3) can be shown summing over all possible such subwords.

Finally, using a small variation of Lemma \ref{rndprojsmall} that can be proven with the same techniques (and fewer technical details), one can show (4). More precisely, one considers separately the case when $w_1\neq\emptyset$ and $w_1= \emptyset$. To deal with the first case, one checks that Lemma \ref{rndprojsmall} still holds when replacing in the statement ``$\calI^k_{w_0^j}(X_n)=I$'' with the new conditioning event that $X_n$ does not terminate with $w_0$. To deal with the first case, one removes the conditioning event altogether and again checks that the proof goes through.

Lemma \ref{hg} applies to any element represented by a word satisfying the conditions above, so that all such elements are weakly contracting.

\section*{Appendix: Weak Tits Alternative}

In this appendix we sketch the proof of a result implied by Theorem \ref{hypemb} and results in \cite{DGO} (see, e.g., Theorem 2.23), that is to say that subgroups containing (weakly) contracting elements are either virtually cyclic or they contain a free group on two generators. We do so, as explained in the introduction, to show that our techniques can prove Theorem \ref{randomwalksintro} in an almost self-contained way.

\begin{thm}
 [Weak Tits Alternative] Let $G$ be a finitely generated group equipped with the weak path system $(X,\calPS)$. Suppose that $G$ contains a weakly $\calPS$-contracting element. Then either $G$ is virtually cyclic or it contains a free group on two generators.
\end{thm}

The theorem above was known already in many cases. For relatively hyperbolic groups it essentially follows from the proofs in \cite{Xie}, for mapping class groups see \cite{Iv, Mc}, for $Out(F_n)$ see \cite{KLu} (the ``full'' Tits Alternative is also known \cite{BFH1,BFH2}), for right-angled Artin groups it follows from Tits' original result \cite{Ti} in view of the fact that they are linear \cite{HW, DJ} (and standard facts about $CAT(0)$ spaces), and for graph manifold groups one just needs to use some Bass-Serre theory.

\begin{proof}\emph{[Sketch.]}
Let $g_0\in G$ be a weakly contracting element. If $G$ is not virtually cyclic, then there exists $h_0\notin E(g)$. From Lemma \ref{hg} we see that for some appropriate $N$ we have that $h_1=h_0g_0^N$ is again weakly contracting. We claim that for any sufficiently large $m$, $h=h_1^m$ and $g=g_0^m$ generate a free group on two generators.

If $A_g, A_h$ are axes for $g_0,h_0$ (and hence also axes for $g,h$), then the diameters of $\pi_{A_g}(A_h), \pi_{A_h}(A_g)$ are finite. This follows, for example, from the proof of Lemma \ref{metriconcat}, in particular claim (1) and the description of the projection on $A$.

Suppose just for convenience that there exists $p\in A_g\cap A_h$, which can be arranged by taking neighbourhoods. Up to increasing $m$, the translation distance of any non-trivial power of $g$ on $A_g$ is much larger than the diameter of $\{p\}\cup\pi_{A_g}(A_h)$ and vice versa.

In order to show the non-triviality of any element of the form, e.g., $g^{a_1}h^{b_1}\dots g^{a_k}h^{b_k}$, where $k\geq 1$ and $a_i,b_i\neq 0$, one consider the sequence of points $x_i=g^{a_1}h^{b_1}\dots g^{a_i}h^{b_i}p$ and $y_i=g^{a_1}h^{b_1}\dots g^{a_i}h^{b_i}g^{a_{i+1}}p$. The non-triviality of such element (for $m$ large) follows from the fact that the concatenation of special paths from $x_i$ to $y_i$ and from $y_i$ to $x_{i+1}$ forms a quasi-geodesic (and that the constants do not depend on $m$). This fact can be proven using the same strategy as in Lemma \ref{metriconcat}: $x_i,y_i$ belong to an appropriate translate $A_i$ of $A_g$, and one can show that $\pi_{A_i}(x_j), \pi_{A_i}(y_j)$ are close to $y_i$ for $j>i$ and close to $x_i$ for $j<i$. One can show this inductively, also considering translates of $A_h$ containing $y_j, x_{j+1}$.
\end{proof}

\bibliographystyle{alpha}
\bibliography{contr.bib}

\def\cprime{$'$}\def\polhk\#1{\setbox0=\hbox{\#1}{øoalign{\hidewidth
  \lower1.5ex\hbox{`}\hidewidth\crcr\unhbox0}}}
\begin{thebibliography}{{Sis}13b}

\bibitem[AK11]{A-K}
Yael Algom-Kfir.
\newblock {Strongly contracting geodesics in outer space}.
\newblock {\em Geom. Topol.}, 15(4):2181–2233, 2011.

\bibitem[Bal95]{Ba-lect}
Werner Ballmann.
\newblock {\em {Lectures on spaces of nonpositive curvature}}, volume~25 of
  {\em {DMV Seminar}}.
\newblock Birkhäuser Verlag, Basel, 1995.
\newblock With an appendix by Misha Brin.

\bibitem[BB08]{BB-rrcconj}
Werner Ballmann and Sergei Buyalo.
\newblock {Periodic rank one geodesics in {H}adamard spaces}.
\newblock In {\em {Geometric and probabilistic structures in dynamics}}, volume
  469 of {\em {Contemp. Math.}}, page 19–27. Amer. Math. Soc., Providence,
  RI, 2008.

\bibitem[BBF10]{BBF}
M.~{Bestvina}, K.~{Bromberg}, and K.~{Fujiwara}.
\newblock {Constructing group actions on quasi-trees and applications to
  mapping class groups}.
\newblock {\em ArXiv e-prints}, June 2010.

\bibitem[BC12]{BC-raagcones}
Jason Behrstock and Ruth Charney.
\newblock {Divergence and quasimorphisms of right-angled {A}rtin groups}.
\newblock {\em Math. Ann.}, 352(2):339–356, 2012.

\bibitem[Beh06]{Be-asgeommcg}
Jason~A. Behrstock.
\newblock {Asymptotic geometry of the mapping class group and {T}eichmüller
  space}.
\newblock {\em Geom. Topol.}, 10:1523–1578, 2006.

\bibitem[BF02]{BeFu-wpd}
Mladen Bestvina and Koji Fujiwara.
\newblock {Bounded cohomology of subgroups of mapping class groups}.
\newblock {\em Geom. Topol.}, 6:69–89 (electronic), 2002.

\bibitem[BF09]{BF}
Mladen Bestvina and Koji Fujiwara.
\newblock {A characterization of higher rank symmetric spaces via bounded
  cohomology}.
\newblock {\em Geom. Funct. Anal.}, 19(1):11–40, 2009.

\bibitem[BF10]{BF1}
Mladen Bestvina and Mark Feighn.
\newblock {A hyperbolic {${\rm Out}(F_n)$}-complex}.
\newblock {\em Groups Geom. Dyn.}, 4(1):31–58, 2010.

\bibitem[BF11]{BF2}
M.~{Bestvina} and M.~{Feighn}.
\newblock {Hyperbolicity of the complex of free factors}.
\newblock {\em
  ArXiv:\href{http://arxiv.org/abs/1107.3308}{\tt{arXiv:1107.3308}}}, July
  2011.

\bibitem[BFH00]{BFH1}
Mladen Bestvina, Mark Feighn, and Michael Handel.
\newblock {The {T}its alternative for {${\rm Out}(F_n)$}. {I}. {D}ynamics of
  exponentially-growing automorphisms}.
\newblock {\em Ann. of Math. (2)}, 151(2):517–623, 2000.

\bibitem[BFH05]{BFH2}
Mladen Bestvina, Mark Feighn, and Michael Handel.
\newblock {The {T}its alternative for {${\rm Out}(F_n)$}. {II}. {A} {K}olchin
  type theorem}.
\newblock {\em Ann. of Math. (2)}, 161(1):1–59, 2005.

\bibitem[BK95]{BK}
S.~V. Buyalo and V.~L. Kobel{\cprime}ski{\u\i}.
\newblock {Geometrization of graph-manifolds. {II}. {I}sometric
  geometrization}.
\newblock {\em Algebra i Analiz}, 7(3):96–117, 1995.

\bibitem[Bow06]{Bow-cchyp}
Brian~H. Bowditch.
\newblock {Intersection numbers and the hyperbolicity of the curve complex}.
\newblock {\em J. Reine Angew. Math.}, 598:105–129, 2006.

\bibitem[Bow12]{Bow-99-rel-hyp}
B.~H. Bowditch.
\newblock {Relatively hyperbolic groups}.
\newblock {\em Internat. J. Algebra Comput.}, 22(3):1250016, 66, 2012.

\bibitem[CS11]{CS-rrc}
Pierre-Emmanuel Caprace and Michah Sageev.
\newblock {Rank rigidity for {CAT}(0) cube complexes}.
\newblock {\em Geom. Funct. Anal.}, 21(4):851–891, 2011.

\bibitem[DGO11]{DGO}
F.~Dahmani, V.~Guirardel, and D.~Osin.
\newblock {Hyperbolically embedded subgroups and rotating families in groups
  acting on hyperbolic spaces}.
\newblock {\em Preprint
  \href{http://arxiv.org/abs/1111.7048}{\tt{arXiv:1111.7048}}}, 2011.

\bibitem[DJ00]{DJ}
Michael~W. Davis and Tadeusz Januszkiewicz.
\newblock {Right-angled {A}rtin groups are commensurable with right-angled
  {C}oxeter groups}.
\newblock {\em J. Pure Appl. Algebra}, 153(3):229–235, 2000.

\bibitem[DMS10]{DMS-div}
Cornelia Dru{\c{t}}u, Shahar Mozes, and Mark Sapir.
\newblock {Divergence in lattices in semisimple {L}ie groups and graphs of
  groups}.
\newblock {\em Trans. Amer. Math. Soc.}, 362(5):2451–2505, 2010.

\bibitem[DS05]{DS-treegr}
C.~Dru{\c{t}}u and M.~Sapir.
\newblock {Tree-graded spaces and asymptotic cones of groups}.
\newblock {\em Topology}, 44(5):959–1058, 2005.
\newblock With an appendix by D. Osin and Sapir.

\bibitem[Gro87]{Gro-hyp}
M.~Gromov.
\newblock {Hyperbolic groups}.
\newblock In {\em {Essays in group theory}}, volume~8 of {\em {Math. Sci. Res.
  Inst. Publ.}}, page 75–263. Springer, New York, 1987.

\bibitem[Ham08]{Ha-isomhyp}
Ursula Hamenstädt.
\newblock {Bounded cohomology and isometry groups of hyperbolic spaces}.
\newblock {\em J. Eur. Math. Soc. (JEMS)}, 10(2):315–349, 2008.

\bibitem[HW99]{HW}
Tim Hsu and Daniel~T. Wise.
\newblock {On linear and residual properties of graph products}.
\newblock {\em Michigan Math. J.}, 46(2):251–259, 1999.

\bibitem[Iva84]{Iv}
N.~V. Ivanov.
\newblock {Algebraic properties of the {T}eichmüller modular group}.
\newblock {\em Dokl. Akad. Nauk SSSR}, 275(4):786–789, 1984.

\bibitem[KL98]{KL}
M.~Kapovich and B.~Leeb.
\newblock {{$3$}-manifold groups and nonpositive curvature}.
\newblock {\em Geom. Funct. Anal.}, 8(5):841–852, 1998.

\bibitem[KL11]{KLu}
Ilya Kapovich and Martin Lustig.
\newblock {Stabilizers of {$\mathbb{R}$}-trees with free isometric actions of
  {$F_N$}}.
\newblock {\em J. Group Theory}, 14(5):673–694, 2011.

\bibitem[Lee95]{L}
Bernhard Leeb.
\newblock {{$3$}-manifolds with(out) metrics of nonpositive curvature}.
\newblock {\em Invent. Math.}, 122(2):277–289, 1995.

\bibitem[Mah11a]{Ma-expdecay}
J.~Maher.
\newblock {Exponential decay in the mapping class groups}.
\newblock {\em To appear in Jour. LMS}, 2011.

\bibitem[Mah11b]{Ma1}
Joseph Maher.
\newblock {Random walks on the mapping class group}.
\newblock {\em Duke Math. J.}, 156(3):429–468, 2011.

\bibitem[McC85]{Mc}
John McCarthy.
\newblock {A “{T}its-alternative” for subgroups of surface mapping class
  groups}.
\newblock {\em Trans. Amer. Math. Soc.}, 291(2):583–612, 1985.

\bibitem[MM99]{MM1}
H.~A. Masur and Y.~N. Minsky.
\newblock {Geometry of the complex of curves. {I}. {H}yperbolicity}.
\newblock {\em Invent. Math.}, 138(1):103–149, 1999.

\bibitem[MM00]{MM2}
H.~A. Masur and Y.~N. Minsky.
\newblock {Geometry of the complex of curves. {II}. {H}ierarchical structure}.
\newblock {\em Geom. Funct. Anal.}, 10(4):902–974, 2000.

\bibitem[MS13]{MS}
Justin Malestein and Juan Souto.
\newblock {On Genericity of Pseudo-Anosovs in the Torelli Group}.
\newblock {\em International Mathematics Research Notices},
  2013(6):1434–1449, 2013.

\bibitem[{Osi}13]{Os-acyl}
D.~{Osin}.
\newblock {Acylindrically hyperbolic groups}.
\newblock {\em
  ArXiv:\href{http://arxiv.org/abs/1304.1246}{\tt{arXiv:1304.1246}}}, April
  2013.

\bibitem[Riv08]{Ri1}
Igor Rivin.
\newblock {Walks on groups, counting reducible matrices, polynomials, and
  surface and free group automorphisms}.
\newblock {\em Duke Math. J.}, 142(2):353–379, 2008.

\bibitem[Riv10]{Ri2}
Igor Rivin.
\newblock {Zariski density and genericity}.
\newblock {\em Int. Math. Res. Not. IMRN}, (19):3649–3657, 2010.

\bibitem[{Sis}11]{3mancones}
A.~{Sisto}.
\newblock {3-manifold groups have unique asymptotic cones}.
\newblock {\em
  ArXiv:\href{http://arxiv.org/abs/1109.4674}{\tt{arXiv:1109.4674}}}, September
  2011.

\bibitem[{Sis}12]{Si-metrrh}
A.~{Sisto}.
\newblock {On metric relative hyperbolicity}.
\newblock {\em
  ArXiv:\href{http://arxiv.org/abs/1210.8081}{\tt{arXiv:1210.8081}}}, October
  2012.

\bibitem[Sis13a]{Si-proj}
A.~Sisto.
\newblock {Projections and relative hyperbolicity}.
\newblock {\em Enseign. Math. (2)}, 59(1-2):165–181, 2013.

\bibitem[{Sis}13b]{Si-hypembqconv}
A.~{Sisto}.
\newblock {Quasi-convexity of hyperbolically embedded subgroups}.
\newblock {\em
  ArXiv:\href{http://arxiv.org/abs/1310.7753}{\tt{arXiv:1310.7753}}}, October
  2013.

\bibitem[{Tio}12]{Ti}
G.~{Tiozzo}.
\newblock {Sublinear deviation between geodesics and sample paths}.
\newblock {\em
  ArXiv:\href{http://arxiv.org/abs/1210.7352}{\tt{arXiv:1210.7352}}}, October
  2012.

\bibitem[Woe00]{Wo}
W.~Woess.
\newblock {\em {Random walks on infinite graphs and groups}}, volume 138 of
  {\em {Cambridge Tracts in Mathematics}}.
\newblock Cambridge University Press, Cambridge, 2000.

\bibitem[Xie07]{Xie}
Xiangdong Xie.
\newblock {Growth of relatively hyperbolic groups}.
\newblock {\em Proc. Amer. Math. Soc.}, 135(3):695–704, 2007.

\end{thebibliography}

\end{document}